\documentclass[10pt]{amsart}
\usepackage{a4wide}
\usepackage{amsmath,amssymb,url,stmaryrd,enumerate,bbm}
\usepackage[francais,english]{babel}
\usepackage[latin1]{inputenc}
\usepackage[all,2cell]{xy}
\usepackage{graphicx}
\usepackage{epic,eepic}
\usepackage{color}

\usepackage{xr}

\usepackage{todonotes}

\usepackage{hyperref}

\usepackage{mathrsfs}

\usepackage{yhmath}

\def\R{\mathbb{R}}

\def\P{\mathbb{P}}
\def\C{\mathbb{C}}

\def\Cp{\mathbb{C}_p}
\def\Fp{\mathbb{F}_p}
\def\Fq{\mathbb{F}_q}
\def\Fqb{\overline{\mathbb{F}}_q}
\def\Fpb{\overline{\mathbb{F}}_p}
\def\L{\mathscr{L}}

\def\A{\text{{\bf A}}}   

\def\O{\mathcal{O}} 
\def\F{\mathscr{F}}

\def\*{^\times }
\def\dpt{\displaystyle}

\def\l{\lambda}

\def\a{\alpha}
\def\b{\beta}

\def\s{\sigma}
\def\ph{\varphi}

\def\e{\varepsilon}

\def\lssi{\Longleftrightarrow}

\def\drt{\rightarrow}
\def\ldrt{\longrightarrow}
\def\Q{\mathbb{Q}}
\def\Qb{\overline{\mathbb{Q}}}

\def\Qlb{\overline{\mathbb{Q}}_\ell}
\def\Qp{\mathbb{Q}_p}
\def\Qpb{\overline{\mathbb{Q}}_p}
\def\Zp{\mathbb{Z}_p}

\def\Z{\mathbb{Z}}
\def\N{\mathbb{N}}

\def\Hom{\text{Hom}}
\def\hom{\mathscr{H}om}

\def\Rpsi{\text{R}\Psi_{\bar{\eta}}}

\def\Gal{\text{Gal}}
\def\={\! = \!}

\def\spec{\text{Spec}}
\def\spf{\text{Spf}}

\def\E{\mathscr{E}}

\def\limp{\underset{\longleftarrow}{\text{ lim }}\;}
\def\limi{\underset{\longrightarrow}{\text{ lim }}\;}

\def\iso{\xrightarrow{\;\sim\;}}

\def\Aut{\text{Aut}}
\def\GL{\hbox{GL}}

\def\xrig{\xrightarrow}

\def\M{\mathcal{M}}

\def\GG{\Gamma}

\def\bc{\backslash}

\def\spa{\text{Spa}}

\def\Lie{\text{Lie}}

\def\Fil{\mathrm{Fil}}

\def\unp{ \big [ {\textstyle\frac{1}{p}}\big ]}

\def\et{\rm{\acute{e}t}}

\def\Sh{\mathrm{Sh}}

\def\div{\mathrm{div}}

\def\DD{\mathbb{D}}
\def\unpi{\big [\tfrac{1}{\pi}\big ]}

\def\LT{\mathcal{LT}}

\def\<<{\langle\langle}
\def\>>{\rangle\rangle}
\def\phmod{\ph\text{-Mod}}

\def\Rep{\text{Rep}}

\def\llparent{( \! ( }
\def\rrparent{) \! ) }
\def\Perf{\text{Perf}}
\def\Eb{\overline{E}}

\def\Hecke{\text{Hecke}}
\def\SL{\text{SL}}
\def\Bun{\text{Bun}}
\def\Gr{\text{Gr}}
\def\hg{\overset{\leftarrow}{h}}
\def\hd{\overset{\rightarrow}{h}}
\def\Fl{\mathcal{F}\ell}
\def\Qpbreve{\breve{\Q}_p}
\def\Sht{\text{Sht}}

\newcommand*{\longhookrightarrow}{\ensuremath{\lhook\joinrel\relbar\joinrel\rightarrow}}

\entrymodifiers={+!!<0pt,\fontdimen22\textfont2>}

 \DeclareMathSymbol{B}{\mathalpha}{operators}{`B}

\author{Laurent Fargues}
\address{Laurent Fargues, CNRS, Institut de Math\'ematiques de Jussieu, 4 place Jussieu 75252 Paris}
\email{laurent.fargues@imj-prg.fr}
\thanks{The author is supported by the project  ANR-14-CE25 "PerCoLaTor"}

\begin{document}

\title{Geometrization of the local Langlands correspondence: an overview}

\date{\today}

\maketitle

\selectlanguage{english}

\begin{abstract}
This article is an overview of the geometrization conjecture for the local Langlands correspondence formulated by the author.
\end{abstract}

\newtheorem{theo}{Theorem}[section]

 \newtheorem{prop}[theo]{Proposition}

\newtheorem{coro}[theo]{{Corollary}}

\newtheorem{lemme}[theo]{Lemma}
\newtheorem{question}[theo]{Question}

\newtheorem{defi}[theo]{Definition}
\newtheorem{exem}[theo]{Example}
\newtheorem{conj}[theo]{Conjecture}
\newtheorem{quest}[theo]{Question}
\newtheorem{Hope}[theo]{Hope}

\theoremstyle{remark}
\newtheorem{rema}[theo]{Remark}

\setcounter{tocdepth}{1}
\tableofcontents

\selectlanguage{english}

\section*{Introduction}

In this article we give an informal exposition of a conjecture 
linking $p$-adic Hodge theory, the geometric Langlands program and the local Langlands correspondence. A more precise but technical article will appear later
(\cite{Geometrisation}). 
\\

Fix a quasi-split reductive group $G$ over a local field $E$, either $E=\Fq\llparent \pi\rrparent$ or $[E:\Qp]<+\infty$, $\Fq=\O_E/\pi$. Fix a discrete Langlands parameter $\ph:W_E\drt \, ^L G$ where $\,^L G$ is the $\ell$-adic Langlands dual of $G$, $\ell\neq p$.
This conjecture associates to $\ph$ a $S_\ph$-equivariant Hecke eigensheaf $\F_\ph$ on the stack of $G$-bundles over the curve we constructed and studied in our joint work with Fontaine (\cite{Courbe}). 

More precisely, given an $\Fq$-perfectoid space $S$ one can construct an $E$-adic space 
$$
X_S
$$
that one has to think of as being the family of curves $(X_{k(s)})_{s\in S}$ where $X_{k(s)}$ is the adic version of the fundamental curve of $p$-adic Hodge theory associated to the perfectoid field $k(s)$ (\cite{Courbe},\cite{ConfLaumon}). When $S=\spa (R,R^+)$ is affinoid perfectoid this has a schematical counterpart as an $E$-scheme equipped with a GAGA equivalence between $G$-bundles on this $E$-scheme and $G$-bundles on $X_S$. 
\\

The stack of $G$-bundles is the stack 
$$
\Bun_G: S\longmapsto \text{groupoid of }G\text{-bundles on }X_S.
$$
This is a stack for the pro-étale topology. In section \ref{sec:defi and prop of Bung} we explain how one should be able to put a geometric structure on this stack via a system of perfectoid charts locally for the "smooth" topology. This notion of a smooth morphism is still vague but we try to convince the reader (and ourselves) that the Quot schemes techniques could be applied in our context. From this point of view, the analog of algebraic spaces is the notion of diamond introduced by Scholze (\cite{ScholzeBerkeley}). One hopes to have a good notion of perverse $\ell$-adic sheaves in this geometric context. 
There is an indication that such a theory of perverse sheaves may exists thanks to recent work of Caraiani-Scholze (\cite{CaraianiScholze}), see the local/global compatibility part of our conjecture in sec.\ref{sec:localglobal}. 
There is another point of view to put a geometric structure on $\Bun_G$ via Beauville-Laszlo uniformization, see sec. \ref{sec:BL}.
\\

Let us note
\begin{itemize}
\item 
 $\spa (E)^\diamond=\spa (E)$ when $E=\Fq\llparent\pi\rrparent$ 
\item  $\spa (E)^\diamond$ is the sheaf of untilts as defined by Scholze (\cite{ScholzeBerkeley}) when $E|\Qp$.
\end{itemize}
Any morphism $S\drt \spa (E)^\diamond$  defines a Cartier divisors $$S^\sharp \hookrightarrow X_S$$ where $S^\sharp=S$ if $E=\Fq\llparent \pi\rrparent$ and $S^\sharp$ is the corresponding untilt of $S$ when $E|\Qp$. When $S=\spa (R,R^+)$, the formal completion of $X_S$ along this Cartier divisor is the formal spectrum of Fontaine's ring $B^+_{dR}(R^\sharp)$. All of this allows us to define Hecke correspondences 
$$
\xymatrix{
& \Hecke^{\leq \mu} \ar[rd]\ar[ld] \\
\Bun_G & & \Bun_G\times \spa (E)^\diamond
}
$$
where the right hand map is a locally trivial fibration in a closed Schubert cell associated to $\mu$ in the $B_{dR}$-affine Grassmanian associated to $G$ (\cite{ScholzeBerkeley}). When $E=\Fq\llparent\pi\rrparent$ this $B_{dR}$-affine Grassmanian is the rigid analytic counterpart of Pappas-Rapoport twisted affine Grassmanian (\cite{Rapoport_Twisted}). The conjecture then says that {\it $\F_\ph$ should be an eigensheaf for those Hecke correspondences with eigenvalue the local system $r_\mu\circ \ph$} where $r_\mu$ is the usual representation of $\,^L G$ dualy associated to $\mu$.
\\

If $B(G)$ is Kottwitz set of $\s$-conjugacy classes in $G( \widehat{E^{un}})$ (\cite{Ko2}) the main theorem of \cite{Gtorseurs} translates in a bijection 
$$
B(G)\iso \big |\Bun_{G,\Fqb}\big |.
$$
This means concretely the following: let $\E$ be a $G$-bundle over $X_S$ with $S$ a perfectoid space over $\Fqb$. Then for each geometric point  $\bar{s}\drt S$, $\E_{|X_{k(\bar{s})}}$ is classified by an element of Kottwitz set $B(G)$. This defines an application $|S|\drt B(G)$ classifying 
the restriction of $\E$ to each geometric fiber.

If $[b]\in B(G)$ is basic then $J_b$, the $\s$-centralizer of $b$, is an inner form of $G$. Moreover the residual gerb at the corresponding point of $\Bun_{G,\Fqb}$ is isomorphic is the classifying space of pro-étale $J_b(E)$-torsors. As a consequence, the stalk of $\F_\ph$ at this point  gives a smooth $\Qlb$-representation of $J_b(E)$ together with a commuting action of $S_\ph$ coming from the action of $S_\ph$ on $\F_\ph$. The conjecture then says that {\it this linear action of $S_\ph$ on this stalk cuts out an L-packet of representations of $J_b(E)$ for a local Langlands correspondence associated to the extended pure inner form $J_b$ of $G$}.

Such extended pure inner forms show up in the work of Kaletha (\cite{Kaletha_Rigid_Inner}, \cite{Kaletha_Supercuspidal} for example). They generalize Vogan's pure inner forms via the embedding $H^1(E,G)\subset B(G)$. One of the motivations for this conjecture comes from the fact 
that after pullback to the curve those reductive group schemes $J_b$ become pure inner forms of $G$ and we thus fall back in Vogan's context.
\\

Let us mention another property of the sheaf $\F_\ph$, the so called {\it character sheaf property}. The fact is that $\F_\ph$ should not be a perverse sheaf on $\Bun_G$ but rather a Weil-perverse sheaf on $\Bun_G\otimes \Fqb$. There is an application $\{G(E)\}_{ell}\drt B(G)_{basic}$. If $\delta \in G(E)$ is such an elliptic element then 
the corresponding point $x_\delta$ of $\Bun_G$ is defined over $\Fq$ and thus the preceding stalk $x_\delta^*\F_\ph$ should be equipped with a Frobenius action. The character sheaf property implies that the function
\begin{eqnarray*}
\{G(E)\}_{ell-reg} & \ldrt & \Qlb \\
\{\delta\} & \longmapsto & \text{Tr} \big ( Frob ; x_\delta^*\F_\ph \big )
\end{eqnarray*}
is the restriction to the elliptic regular subset of the stable distribution of $G(E)$ associated to $\ph$.
\\

The structure of this article is the following:
\begin{itemize}
\item In sections \ref{sec:The curve}, \ref{sec:The stack BunG} and \ref{sec:The BdR affine Grassmanian and the Hecke stack} we explain the construction and the geometry of the objects showing up in our conjecture: the curve $X_S$, the stack $\Bun_G$, the associated Hecke stack and the $B_{dR}$-affine grassmanian. At some point we recall quickly some constructions of Scholze (\cite{ScholzeBerkeley}), for example his theory of diamonds.
\item  In section \ref{sec:Statement of the conjecture} we give the precise statement of the conjecture. 
\item The purpose of section \ref{sec:The character sheaf property} is to explain in more details the character sheaf property that appears in the conjecture.
\item Section \ref{sec:isocris BT et vect} is a preliminary section for sections \ref{sec:localglobal} and \ref{sec:Kottwitz conj}. We explain that 
the Hodge filtration of the $F$-isocrystal of a $p$-divisible group in unequal characteristic defines a morphism of pre-stacks from the pre-stack of $p$-divisible groups toward the Hecke stack. From this point of view the 
two morphisms $\hg$ and $\hd$ that define the Hecke stack as a correspondence can be interpreated as the de-Rham period morphism and the Hodge-Tate period morphism.
\item In section \ref{sec:localglobal} we explain the local/global compatibility part of the conjecture. This says there is a compatibility between our conjectural local perverse sheaf $\F_\ph$ and Caraiani-Scholze perverse sheaf $R\pi_{HT*}\Qlb$ constructed from the relative cohomology of a Hodge type Shimura variety over its Hodge-Tate period space.
\item In section \ref{sec:Kottwitz conj} we explain why our conjecture implies Kottwitz conjecture describing the discrete part of the cohomology of moduli spaces of $p$-divisible groups as defined by Rapoport and Zink (\cite{RZ}). We begin first by explaining the Lubin-Tate/Drinfeld case (sec.\ref{sec:Drinfeld and Lubin Tate case}): the study of this key case is the main ingredient that lead the author to his conjecture. 
\item In section \ref{sec:The abelian case} we verify the conjecture for tori. In this case the conjecture is implied by local class field theory. 
\end{itemize}

\vspace{0.8cm}

{\it Thanks: The author would like to thank Kiran Kedlaya, Peter Scholze 
and more generally all the participants of the program "New Geometric Methods in Number Theory and Automorphic Forms" that took place in fall 2014 at the MSRI.
}

\section{The curve}
\label{sec:The curve}

\label{sec:the curve}

\subsection{The adic curve}

Let $E$ be a local field with finite residue field $\Fq=\O_E/\pi$. Thus, either $[E:\Qp]<+\infty$ or $E=\Fq\llparent \pi\rrparent$.
We note $$\text{Perf}_{\Fq}$$
the category of perfectoid spaces over $\Fq$. Let us just note that in this text perfectoid spaces don't live over a fixed base perfectoid field
as this is the case in Scholze's original definition (\cite{Perfectoid}). Here we take the more general definition that can be found in \cite{BourbakiPerfectoid} or \cite{ScholzeBerkeley}.
\\

Given $S\in \Perf_{\Fq}$ one can the define an $E$-adic space $X_S$ that is, by definition, uniformized by a space $Y_S$
$$
X_S = Y_S/\ph^\Z
$$
where $\ph$ is a Frobenius action properly discontinuously on $Y_S$.

\subsubsection{The equal characteristic case}
\label{sec:the curve equal}

The case when $E=\Fq \llparent\pi\rrparent$ is more simple. In fact we have the following definition.

\begin{defi}
If $E=\Fq \llparent \pi\rrparent$ set
$$
Y_S=\DD^*_S\subset \mathbb{A}^1_S
$$
an {\it open punctured disk of the variable $\pi$} where $\DD^*_S =\{ 0<|\pi|<1\}$. The Frobenius $\ph$ acting on $\DD^*_S$ is the one of $S$.
\end{defi}

 If
$S=\spa (R,R^+)$ is affinoid perfctoid then 
$$
\DD^*_S= \spa (R^\circ\llbracket\pi\rrbracket, R^++\pi R^\circ \llbracket \pi\rrbracket)\setminus V(\pi \varpi_R )
$$
where $\varpi_R\in R^{\circ\circ}\cap R^\times$ is a so called {\it pseudo-uniformizing element} (\cite{BourbakiPerfectoid}). Moreover, $\ph$ is induced by  
$$
\ph\big  (\sum_{n\geq 0} x_n\pi^n\big )=\sum_{n\geq 0} x_n^q\pi^n.
$$
as an automorphism of $R^\circ \llbracket \pi\rrbracket$. 
The choice of $\varpi_R$ defines a radius function
\begin{eqnarray}\label{eq:fonction delta equal char}
\delta:|\DD^*_S| &\ldrt  & ]0,1[ \\
\nonumber
y & \longmapsto & p^{-\frac{\log |\pi (y^{max})|}{\log | \varpi_R (y^{max})|}}
\end{eqnarray}
where $y^{max}$ is the maximal generalization of $y$, a valuation whose value group may be taken to be a subgroup of $\R$ ($y^{max}$ is the Berkovich point associated to $y$ via the identification of the Berkovich spectrum as a quotient of the adic one). This continuous function satisfies $$\delta (\ph (y))=\delta (y)^{1/q}$$ and {\it the action of $\ph$ is thus properly discontinuous without fixed points}. One then has two structural maps 
$$
\xymatrix{
  & \DD^*_S \ar[ld] \ar[rd] \\
S  & & \DD^*_{\Fq}=\spa (E)
}
$$
where:
\begin{itemize}
\item The first one is the standard one of the punctured disk, locally of finite type but {\it  not $\ph$-invariant}.
\item The second one is obtained via the identification of the open punctured disk over $\Fq$ (equipped with the trivial valuation) with $\spa (E)$. This one is $\ph$-invariant but {\it not locally of finite type}. This is the structural morphism we are interested in. 
\end{itemize}

\begin{defi} Set
$$
X_S=\DD^*_S/\ph^\Z
$$
as an $E$-adic space.
\end{defi}

{\it This space $X_S$ is not of finite type over $E$.
Moreover it has  has no structural map to $S$}. Nevertheless, for $\tau\in \{\text{analytic},\text{étale}, \text{pro-étale}\}$ the functor $T/S\mapsto X_T/X_S$
defines a continuous morphism of sites 
$$
(X_S)_\tau \ldrt S_{\tau}.
$$
One can think of this as being a consequence of the fact that even if the space "$S/\text{Frob}^\Z$" makes no sense, the associated site "$(S/\text{Frob}^\Z)_\tau$" has to coincide with $S_\tau$. In particular there is a continuous application 
$$
|X_S|\ldrt |S|
$$
that is quasicompact specializing ($f$ is specializing if whenever $f(x)\succ y$ there exists $x\succ x'$ with $f(x')=y$) and thus in particular closed when $S$ is quasicompact quasiseparated (a spectral map between spectral spaces is closed if and only if it is specializing). 
{\it One can thus think of $X_S$ as being  "proper over $S$".}
\\

When $S$ is the spectrum of a perfectoid field $F$ the space $X_F$ satisfies strong finiteness properties:
\begin{itemize}
\item It is strongly noetherian in Huber's sense
\item  It is "a curve": it is covered by a finite set of spectra of $E$-Banach algebras that are P.I.D.
\end{itemize}
Those properties are clear since for $\DD^*_F$, seen as a classical Tate rigid space over $F$, the algebras of holomorphic functions on closed annuli are P.I.D.. 
{\it But this is not the case anymore for general $S$: the space $X_S$ does not satisfy any reasonable finiteness properties for a general $S$.} That being said, one can think of $X_S$ as being the family of curves
$$
(X_{k(s)})_{s\in S}
$$
although there is no way to give a precise meaning to this.

\subsubsection{The unequal characteristic case}
\label{sec:courbe unequal}

Suppose now $[E:\Qp]<+\infty$. Let $S=\spa (R,R^+)$ be affinoid perfectoid. 
The analog of the preceding ring $R^\circ \llbracket\pi\rrbracket$ is {\it Fontaine's ring $A_{inf}$} noted $\A$.

\begin{defi}
Define  
\begin{eqnarray*}
\A &=& W_{\O_E} (R^\circ) \\
 &=& W(R^\circ)\otimes_{W(\Fq)} \O_E \\
 &=& \Big \{ \sum_{n\geq 0}[x_n]\pi^n \ | \ x_n\in R^\circ \Big \} \ \ \text{ (unique writing)}
\end{eqnarray*}
where the notation $W_{\O_E}$ means the ramified Witt vectors. Set 
$$
\A^+=\Big \{ \sum_{n\geq 0}[x_n]\pi^n\in \A\ |\ x_0\in R^+\Big \}.
$$
\end{defi}

\begin{defi}
Equip $\A$ with the $(\pi,[\varpi_R])$-adic topology and set
$$
Y_{R,R^+}= \spa (\A,\A^+)\setminus V(\pi [\varpi_R]).
$$
\end{defi}

Replacing $\varpi_R$ by $[\varpi_R]$ in the formula (\ref{eq:fonction delta equal char}) one has a continuous function
$$
\delta:|Y_S|\ldrt ]0,1[. 
$$
Let us fix a power multiplicative Banach algebra norm $\|.\|:R\ldrt \R_+$ defining the topology of $R$ (once $\|\varpi_R\|$ is fixed such a norm is unique). {\it The space $Y_{R,R^+}$ is Stein} and thus completely determined by $\O(Y_{R,R^+})$. This is described in the following way. Consider
\begin{eqnarray*}
\O(Y_{R,R^+})^b &=& \A[\tfrac{1}{\varpi_E},\tfrac{1}{[\varpi_R]}] \\
&=& \Big \{ \sum_{n\gg -\infty} [x_n]\pi^n\ |\ x_n \in R,\ \sup_n \|x_n\|<+\infty\Big \}
\end{eqnarray*}
the ring of holomorphic functions on $Y_{R,R^+}$ meromorphic along the divisors $\pi=0$ and $[\varpi_F]=0$. 
For each $\rho\in ]0,1[$ there is a power-multiplicative Gauss norm $\|.\|_\rho$ on $\O(Y_{R,R^+})^b$ defined by
$$
\big \|\sum_{n\gg-\infty} [x_n]\pi^n \big \|_\rho = \sup_{n\in \Z} \|x_n\|\rho^n.
$$
Then $\O(Y_{R,R^+})$ is the $E$-Frechet algebra obtained by completion of $\O(Y_{R,R^+})^b$ with respect to $(\|.\|_\rho)_{\rho\in ]0,1[}$.
\\

The action of Frobenius $\ph$ on $Y_{R,R^+}$ is given by the usual Frobenius of the Witt vectors
$$
\ph ( \sum_{n\geq 0} [x_n]\pi^n)=\sum_{n\geq 0} [x_n^q]\pi^n.
$$
One then checks that {\it the functor $(R,R^+)\mapsto Y_{R,R^+}$ glues to a functor $S\mapsto Y_S$} and defines $X_S=Y_S/\ph^\Z$. Moreover all properties of section \ref{sec:courbe unequal} apply to $X_S$ (although some of them are much more difficult to prove, typically the fact that $X_F$ is "a strongly noetherian curve"  for a perfectoid field $F$, see \cite{Courbe} and \cite{KedlayaFortementNoetherien}).

\subsection{The link between the equal and unequal characteristic case}

{\it In fact the preceding $E$-adic spaces $Y_{S,E}$ and $X_{S,E}=Y_{S,E}/\ph^\Z$ can be defined for any complete non-archimedean field $E$ with residue field $\Fq$.} We are only interested in the case when $E$ has discrete valuation (the preceding case) of $E$ is perfectoid. The fact is that this construction is compatible with scalar extensions:
$$
Y_{S,E'}=Y_{S,E}\hat{\otimes}_E E',\ \ X_{S,E'}=X_{S,E}\hat{\otimes}_E E'.
$$
Of course, when $E|\Qp$ those formulas determine completely $Y_{S,E}$ and $X_{S,E}$ from the case when $E$ has discrete valuation since if $E_0=W(\Fq)\unp$ 
they give
$$
Y_{S,E} = Y_{S,E_0}\hat{\otimes}_{E_0} E,\ \ X_{S,E} = X_{S,E_0}\hat{\otimes}_{E_0} E.
$$
This is not the case anymore when $E$ has unequal characteristic since à priori those formulas give a definition of our spaces that could depend on the choice of a sub-discrete valuation field inside $E$. What we say is that there is still a natural definition for any $E$. To be precise, here is a definition/statement.

\begin{defi}
Consider $E$ with residue field $\Fq$ but no necessarily of discrete valuation.
\begin{enumerate}
\item If $E|\Qp$  set $Y_{S,E}=Y_{S,E_0}\hat{\otimes}_{E_0} E$ and the same for $X_{S,E}$.
\item For $E$ of characteristic $p$ with residue field $\Fq$ and $(R,R^+)$ affinoid perfectoid over $\Fq$ set 
$$
Y_{R,R^+} = \spa ( R^\circ \hat{\otimes}_{\Fq} \O_E) \setminus V ( \varpi_E \varpi_R)
$$
where $\varpi_E$ is a pseudo-uniformizing element in $E$. This glues to a functor $S\mapsto Y_{S,E}$. 
\end{enumerate}
\end{defi}

In fact when $E$ has characteristic $p$ one has the simple formula (whose proof is quite elementary):
\begin{equation}\label{eq:formule produit egale}
Y_{S,E} = S\times_{\spa (\Fq)} \spa (E)
\end{equation}
as a fiber product in the category of (analytic) adic spaces over $\Fq$. {\it This formula, although being quite simple and elegant, does not say anything about the geometry of $Y_{S,E}$} (typically what is a vector bundle on it ?). Let us note that via this formula $$\ph=\text{Frob}_S\times Id$$ and we thus have 
 the simple formula 
\begin{equation}
X_{S,E} = (S\times_{\spa (\Fq)} \spa (E))/\text{Frob}_S^\Z\times Id
\end{equation}
where the quotient is the sheaf quotient for the analytic topology.
\\

Suppose for example that $E=\Fq\llparent \pi^{1/p^\infty}\rrparent$. Then 
\begin{eqnarray*}
Y_{S,E} &=& \underset{\text{Frob}}{\limp} \DD^*_S \\
&=& \DD^*_S\hat{\otimes}_{\Fq\llparent \pi\rrparent} \Fq \llparent \pi^{1/p^\infty}\rrparent
\end{eqnarray*}
the perfectoïd open punctured disk.
\\

We just remarked that when $E=\Fq \llparent\pi^{1/p^\infty}\rrparent$ the space $Y_{S,E}$, and thus $X_{S,E}$, is perfectoid. {\it This is in fact more general in equal and unequal characteristic}.

\begin{prop}
 {\it for any perfectoid $E$ the spaces $Y_{S,E}$ and $X_{S,E}$ are perfectoid}. For such a perfectoid $E$ one has moreover the very simple formulas
$$
Y_{S,E}^\flat = Y_{S,E^\flat},\  \ X_{S,E}^\flat=X_{S,E^\flat}.
$$
\end{prop}

\begin{rema}
The fact that for $E$ discretely valued $Y_{(R,R^+),E}$ becomes perfectoid after scalar extension  to a perfectoid field is the main point to prove that $Y_{(R,R^+),E}$ is 
an adic space i.e. Huber's presheaf is a sheaf. It should thus have been pointed before this section but since this is a review article the author chosed to care less about strict rigor and focus more on exposition.
\end{rema}

Let us treat an example. Take $[E:\Qp]<+\infty$. Let $E_{\infty}|E$ be the completion of the extension generated by the torsion points of a Lubin-Tate group law $\LT$ over $\O_E$.
The Lubin-Tate character induces an isomorphism $\Gal (E_\infty|E)\xrig{\sim}\O_E^\times$. 
 We then have
$$
E_\infty^\flat = \Fq \llparent T^{1/p^\infty}\rrparent 
$$
where if $\e= (\e^{(n)})_{n\geq 0}$ is a generator of the Tate-module of $\LT$, $[\pi]_{\LT} \big ( \e^{(n+1)}\big ) = \e^{(n)}$, then 
\begin{eqnarray*}
T &=& (\e^{(n)} \text{ mod }\pi )_{n\geq 0} \in \underset{Frob}{\limp} \O_{E_{\infty}}/\pi = \O_{E_\infty}^\flat.
\end{eqnarray*}
The action of $a\in \O_E^\times=\Gal (E_\infty|E)$ on $E_\infty^{\flat,\times}$ is then given by $[a]_\LT\in \Aut (\Fq\llbracket T\rrbracket)$. The action {\it $[-]_\LT$ induces an action of $\O_E^\times$ on 
$\DD^*_S$} and there is an $\O_E^\times$-equivariant identification
\begin{eqnarray*}
\big (Y_{S,E}\hat{\otimes}_E E_\infty \big )^\flat &=&
Y_{S,E_\infty}^\flat \\
&=& Y_{S,E_\infty^\flat} \\
&=& \underset{\text{Frob}}{\limp} \DD^*_S. 
\end{eqnarray*}

\subsection{The diamond product formula}\label{sec:diamond product formula}

We will explain that formula (\ref{eq:formule produit egale}) extends to
the unequal characteristic using Scholze's theory of diamonds (\cite{ScholzeBerkeley}). We will moreover use systematically diamonds later in this text, this is why we give a quick review of the theory.

\subsubsection{Diamonds}

Consider the case $[E:\Qp]<+\infty$. We just saw that although $Y_{F,E}$ is not perfectoid, $Y_{F,E}\hat{\otimes} E_\infty$ is perfectoid with tilting $\DD^{*,1/p^\infty}_S$, the perfectoid punctured disk. One is then tempted to {\it go down} and say that $Y_{S,E}$ is 
\begin{eqnarray}\label{eq:premier quotient diamand}
"\DD^{*,1/p^\infty}_S /\O_E^\times".
\end{eqnarray}
This is is resumed in the following {\it diamond} diagram
$$
\xymatrix@C=7mm{
 & Y_{S,E_\infty} \ar[ld]_{(-)^\flat} \\
  \DD^{*,1/p^\infty}_S \ar[rd] && Y_{S,E} \ar[lu]_{-\hat{\otimes} E_\infty} \\
  & \DD^{*,1/p^\infty}_S/\O_E^\times \ar@{-}[ru]_{=}
}
$$
This is a more general phenomenon. Typically consider the torus $\spa (\Qp\langle T^{\pm 1}\rangle)$. It has a pro-Galois perfectoid cover 
$$
\spa (\Qp^{cyc}\langle T^{\pm 1/p^\infty} \rangle )\ldrt \spa (\Qp\langle T^{\pm 1}\rangle).
$$
The same argument as before leads us to write 
\begin{eqnarray}\label{eq:deuxieme quotient diamand}
"\spa (\Qp\langle T^{\pm 1}\rangle) = \spa (\Qp^{cyc,\flat} \langle T^{\pm 1/p^\infty}\rangle )/\Zp (1)\rtimes \Gal (\Qp^{cyc}|\Qp)".
\end{eqnarray}
The theory of diamonds allows us to give a meaning to the formulas (\ref{eq:premier quotient diamand}) and (\ref{eq:deuxieme quotient diamand}).
\\
 
 For this,
equip $\Perf_{\Fq}$ and $\Perf_E$ with the {\it pro-étale topology.}

\begin{defi}
\begin{enumerate}
\item If $S\in \Perf_{\Fq}$, {\it an untilt of $S$} is by definition a couple $(S^\sharp,\iota)$ where $S^\sharp\in \Perf_E$ and $\iota:S\xrig{\sim} S^{\sharp,\flat}$. 
\item If $\F$ is a pro-étale sheaf on $\Perf_E$ define $\F^\diamond$ to be the sheaf on $\Perf_{\Fq}$ defined by
$$
\F^\diamond (S) = \big \{ (S^\sharp,\iota,s)\ |\ (S^\sharp,\iota)\text{ is an untilt of } S\text{ and } s\in \F (S^\sharp)\big \}/\sim.
$$
\end{enumerate}
\end{defi}

\begin{rema}
The fact that $\F^{\diamond}$ is a sheaf relies on Scholze's purity. In fact this is deduced from the fact that tilting induces an equivalence of pro-étale sites $\Perf_{S^{\sharp}}\iso \Perf_S$.
\end{rema}

For example {\it $\spa (E)^\diamond$ is the sheaf  of untilts to $E$} (up to the evident equivalence relation). Most of the time, to simplify the notations, we will say {\it  "$S^\sharp$ is an untilt of $S$"} although one should say we consider a couple $(S^\sharp,\iota)$ up to the evident equivalence relation.
\\

The preceding diamond functor defines an equivalence of sites and topoi
\begin{eqnarray*}
\Perf_{E} & \iso & \Perf_{\Fq}/\spa (E)^\diamond \\
\widetilde{\Perf}_E & \iso & \widetilde{\Perf}_{\Fq}/\spa (E)^\diamond.
\end{eqnarray*}
Of course, if $\F$ is representable by some $T\in \Perf_E$ then $T^\diamond=T^\flat$. 
\\

We then fall on the following question:
 If $X$ is a rigid $E$-space it defines naturally a sheaf on $\Perf_E$. Do we loose some properties of $X$ by looking at this sheaf ? Said in another way: does $X^\diamond/\spa (E)^\diamond$ characterize $X$ ?

The answer is positive for normal rigid spaces thanks to the following two properties.

\begin{prop}
\begin{enumerate}
\item 
 Any  rigid analytic space  has a perfectoid pro-étale cover.
 \item If $X$ is moreover normal and $\nu:X_{\text{pro-ét}}\drt X_{\et}$ is the associated morphism of sites $\nu_*\widehat{\O}_X=\O_X$  where $\widehat{\O}_X (R,R^+)= R$ if $\spa (R,R^+)\drt X$ is pro-étale affinoïde perfectoid (\cite{ScholzeComparaisonHodge}).
\end{enumerate}
\end{prop}

To illustrate those properties consider $X=\spa (E \langle T_1^{\pm 1},\dots, T_d^{\pm 1}\rangle )$. Then 
$$
\spa (\Cp  \langle T_1^{\pm 1/p^\infty},\dots, T_d^{\pm 1/p^\infty}\rangle )\ldrt X
$$
is a a perfectoid pro-étale Galois cover with Galois group $\Zp (1)^d \rtimes \Gal (\overline{E}|E)$. This proves property (1) for smooth rigid spaces over $E$.  
One then checks by an explicit computation that 
$$
\big [ \Cp  \langle T_1^{\pm 1/p^\infty},\dots, T_d^{\pm 1/p^\infty}\rangle \big ]^{\Zp (1)^d \rtimes \Gal (\overline{E}|E)} = E \langle T_1^{\pm 1},\dots, T_d^{\pm 1}\rangle.
$$
This proves property (2) for smooth rigid spaces over $E$.

\begin{coro}
 We  have an embedding of the category of normal rigid spaces over $E$ inside $\widetilde{\Perf}_{\Fq}/\spa (E)^\diamond$.
\end{coro}
 
 But in fact the essential image of the functor $X\mapsto X^\diamond$ is of particular type as we could see for example in formula (\ref{eq:deuxieme quotient diamand}): {\it this is an algebraic space for the pro-étale topology in $\Perf_{\Fq}$}. This is exactly the definition of a diamond. 

\begin{defi}[Scholze \cite{ScholzeBerkeley}]
A diamond is an algebraic space for the pro-étale topology in $\Perf_{\Fp}$. 
\end{defi}

If $X$ is a normal rigid space over $E$ choose some perfectoid pro-étale covering $\widetilde{X}\drt X$. One can check (this needs a little work but is not such difficult) that:
\begin{itemize}
\item 
 the fibre products $\mathcal{R}=\widetilde{X}\times_X \widetilde{X}$ and $\widetilde{X}\times_{\spa (E)} \widetilde{X}$ exist as perfectoid spaces with
  $(\widetilde{X}\times_{\spa (E)} \widetilde{X})^\flat= \widetilde{X}^\flat\times_{\spa (\Fq)} \widetilde{X}^\flat$. 
\item The projections $\xymatrix{\mathscr{R}\ar@<.8ex>[r] \ar@<-.8ex>[r] & \widetilde{X}}$ are pro-étale.
\end{itemize}  
   Then $\mathscr{R}^\flat$ is a pro-étale equivalence relation on $\widetilde{X}^\flat$ and 
$$
X^{\diamond} = \widetilde{X}^\flat/\mathscr{R}^\flat.
$$
We can thus state the following.

\begin{coro}
There is a fully faithfull functor 
\begin{eqnarray*}
\text{Normal rigid spaces over }E &\ldrt & \text{Diamonds over } E^\diamond \\
X & \longmapsto & X^{\diamond}/E^\diamond.
\end{eqnarray*}
\end{coro}

\subsubsection{The product formula}

{\it The same principles apply to pre-perfectoid spaces over $E$}: $E$-adic spaces that become perfectoid after extensions to any perfectoid extension of $E$. This is the case for example of $\spa (E\langle T^{\pm 1/p^\infty}\rangle )$ or $Y_{S,E}$. We can now compute $Y_{S,E}^\diamond$:
\begin{eqnarray*}
Y_{S,E}^{\diamond} &=& Y_{S,E_\infty^\flat} /\O_E^\times \\
&=& S\times \spa (E_\infty^\flat)/\O_E^\times \\
&=& S\times \spa (E)^\diamond. 
\end{eqnarray*}
We thus find a formula analog to (\ref{eq:formule produit egale}). We thus have at the end 
$$
X_{S,E}^\diamond = (S\times \spa (E)^\diamond)/\ph^\Z
$$
where $\ph= \text{Frob}_S\times Id$.

\begin{rema}
The computation of $Y_{S,E}^\diamond$ can in fact be done without using the auxiliary extension $E_\infty$. Digging a little bit on finds this is a consequence of the fact that we have two adjoint functors
$$
\xymatrix{
\text{Perfect }\Fp\text{-algebras} \ar@<1ex>[r]^-{W(-)} & \ar@<1ex>[l]^-{(-)^\flat} p\text{-adic rings}
}
$$
where the adjunction morphisms are Fontaine's theta $\theta: W(A^\flat)\drt A$ and $R\xrig{\sim} W(R)^\flat$ given by $x\mapsto \big ( \big [ x^{p^{-n}}\big ]\big )_{n\geq 0}$.
\end{rema}

\subsection{Untilts as Cartier divisors on the curve}

In \cite{Courbe} the authors proved that the closed points on the curve they constructed associated to a perfectoid field $F|\Fq$ corresponds to untilts of $F$ up to powers of Frobenius. This generalizes in the relative setting.

\subsubsection{The equal characteristic case}

Suppose $E=\Fq \llparent \pi\rrparent$. Let $S\in \Perf_{\Fq}$. Since $Y_S=S\times \spa (E)$, any morphism $S\drt \spa (E)$ induces a section of the structural morphism 
$Y_{S}\drt \spa (E)$. This is in fact a closed embedding 
$$
S\longhookrightarrow Y_S=\DD^*_S
$$
locally defined by $(\pi-a)$ where $a\in R^{\circ\circ}\cap R^\times$ if $S=\spa (R,R^+)$. One thinks about it as being a degree $1$ Cartier divisor on $Y_S$. 

\begin{defi}
an element $f=\sum_{n\geq 0} a_n\pi^n\in R^\circ \llbracket\pi\rrbracket$ is primitive of degree $1$ if $a_0\in R^{\circ\circ}\cap R^\times$ and $a_1\in R^{\circ \times}$.
\end{defi}

 According to {\it Weierstrass factorization} any degree $1$ primitive element in $R^\circ \llbracket\pi\rrbracket$ can be written as a unit times $\pi-a$ for some uniquely defined $a\in R^{\circ\circ}\cap R^\times$. We deduce the following lemma.

\begin{lemme}
There is  a bijection
\begin{eqnarray*}
&& \big \{ \text{morphisms } S\drt \spa (E)\big \} \\ &\simeq & \big \{ \text{closed immersions } T\hookrightarrow Y_S\text{ locally defined by a degree 1 primitive element}\big \}.
\end{eqnarray*}
\end{lemme}

\subsubsection{The unequal characteristic case}

Suppose now $[E:\Qp]<+\infty$. Let $R$ be a perfectoid $\Fq$-algebra. 

\begin{defi}
  $f=\sum_{n\geq 0}[a_n]\pi^n\in \A=W_{\O_E}(R^\circ)$ is a degree $1$ primitive element if $a_0\in R^{\circ\circ}\cap R^\times$ and $a_1\in R^{\circ\times}$.
\end{defi}
  
   {\it The preceding statement about Weierstrass factorization is then not true anymore.} 
     Nevertheless we have the following.
   
\begin{prop}   
 \begin{enumerate}
 \item If $R^\sharp$ is an untilt of $R$ over $E$ then the kernel of Fontaine's $\theta:W_{\O_E}(R^\circ)\twoheadrightarrow R^{\sharp,\circ}$ is generated by a degree $1$ primitive element.
 \item Reciprocally, if $f\in W_{\O_E}(R^\circ)$ is primitive then $W_{\O_E} (R^{\circ})\unpi/f$ is a perfectoid $E$-algebra that is an untilt of $R$.
\end{enumerate}  
\end{prop}

\begin{exem}
Take $F=\Qp^{cyc,\flat} =\Fp \llparent T^{1/p^\infty}\rrparent$ with $T=\big ( ( \zeta_{p^n}-1)\text{ mod }p\big )_{n\geq 0}\in \underset{\text{Frob}}{\limp} \Zp^{cyc}/p=\Zp^{cyc,\flat}$. Then $\Zp^{cyc}= W ( \Fp\llbracket T^{1/p^\infty}\rrbracket)/ 1+[1+T^{1/p}]+\dots + [1+T^{\frac{p-1}{p}}]$. In this case the primitive element 
$1+[1+T^{1/p}]+\dots+ [1+T^{\frac{p-1}{p}}]$ has no Weierstrass factorization as a unit times $p-[a]$ for some $a$ since $p$ has no $p$-root in $\Qp^{cyc}$.
\end{exem}

\begin{rema}
If $R^\sharp= W_{\O_E} ( R^\circ)\unpi/f$ then after a pro-étale covering of $R$ the primitive element $f$ has a Weierstrass factorization as a unit times $\pi-[a]$. In fact, it suffices to tilt the pro-étale covering of $R^\sharp$ given by the completion of  $\cup_{n\geq 0} R^\sharp (\pi^{1/p^n})$. Nevertheless {\it such an $a$ is not unique like in the equal characteristic case}.
\end{rema}
 
 From this we deduce this corollary.
 
 \begin{coro}
 We  have a bijection 
 $$
 \spa (E)^\diamond (R,R^+) \simeq \big \{ \text{degree 1 primitive elements in }W_{\O_E}(R^\circ) \big \}/W_{\O_E}(R^\circ )^\times.
 $$
 \end{coro}
 
Now, the formula $Y_{S,E}^\diamond =S\times \spa (E)^\diamond$ tells us that any untilt $S^\sharp$ of $S$ defines a morphism $S=S^{\sharp,\diamond}\drt Y_{S,E}^\diamond$ that is to say a morphism $S^{\sharp}\drt Y_{S,E}$ that is in fact a closed immersion locally defined by a degree $1$ primitive element. We thus have:

\begin{prop}
There is a bijection 
\begin{eqnarray*}
&& \big \{ \text{untilts of }S\big \} \\
&\simeq & \big \{ \text{closed immersions } T\hookrightarrow Y_S\text{ locally defined by a degree 1 primitive element}\big \}.
\end{eqnarray*} 
\end{prop}

Moreover on checks that if $S^\sharp$ is an untilt of $S$ then the closed immersion $S^\sharp \hookrightarrow Y_S$ induces a closed immersion
$$
S^\sharp \longhookrightarrow X_S.
$$
There is a bijective action of Frobenius on untilts of $S$ via the formula $(S^\sharp,\iota)\mapsto (S^\sharp,\iota\circ \text{Frob})$. This is given by the action of $\ph$ on the set of degree $1$ primitive elements. We then have the following proposition.

\begin{prop}
There is a map
\begin{eqnarray*}
&& \big \{ \text{untilts of }S\big \}/\text{Frob}^\Z \\
&\ldrt & \big \{ \text{closed immersions } T\hookrightarrow X_S\text{ locally defined by a degree 1 primitive element}\big \}.
\end{eqnarray*} 
\end{prop}

\subsubsection{Some Hilbert diamonds}

Finally let us say that for any $d\geq 1$, there is a notion of degree $d$ primitive element. More precisely, $f=\sum_{n\geq 0}[a_n]\pi^n \in \A$ is primitive of degree $d$ if 
$a_0\in R^{\circ\circ}\cap R^\times$, $a_1,\dots,a_{d-1}\in R^{\circ\circ}$ and $a_d\in R^{\circ\times}$. Then degree $d$ primitive elements in $\A$ up to $\A^\times$ are in bijection with 
$$
\big [(\spa (E)^{\diamond})^d/\mathfrak{S}_d \big ](R,R^+)
$$
where the quotient by the symmetric group is for the faithfull topology introduced in \cite{ScholzeBerkeley}. Thus, degree $d$ primitive elements are related to the symmetric product of the curve 
$$
X_S^{[d],\diamond} = (S \times \spa (E)^{\diamond,d})/\ph^{\Z}\times \mathfrak{S}_d.
$$
Those {\it Hilbert-diamonds} seem quite interesting but we won't say more since we don't need them to formulate our conjecture (although there is no doubt they will show up in the proof of the conjecture).

\subsection{The schematical curve}
\subsubsection{Definition}

Let $S=\spa (R,R^+)$ be affinoid perfectoid over $\Fq$. For each $d\in \Z$ there is a line bundle $\O(d)$ on $Y_{R,R^+}$ whose geometric realization is
$$
Y\underset{\ph^\Z}{\times} \mathbb{A}^1\ldrt Y/\ph^\Z
$$
where $\ph$ acts on $\mathbb{A}^1$ via $\times \pi^{-d}$. {\it We now declare $\O(1)$ is ample} via the following definition.

\begin{defi}\label{defi:schematical curve}
\begin{enumerate}
\item Set 
\begin{eqnarray*}
P_R &=& \bigoplus_{d\geq 0} H^0 ( X_{R,R^+}, \O(d)) \\
&=& \bigoplus_{d\geq 0} \O(Y_{R,R^+})^{\ph=\pi^d}.
\end{eqnarray*}
\item Define $X_R=\text{Proj} (P_R)$. We will use the notation $X_R^{sch}$, resp. $X_{R,R^+}^{ad}$, to specify we use the schematical curve, resp. the adic one, when there is an ambiguity.
\end{enumerate}
\end{defi}

The case when $R$ is a perfectoid field $F$ is the object of study of \cite{Courbe}. {\it In this case $X_F$ is a Dedekind scheme}. 
In general it does not satisfy any finiteness properties. 
\\

There is always a morphism of ringed spaces 
$$
X_{R,R^+}^{ad}\ldrt X_R^{sch}.
$$

\subsubsection{A computation in equal characteristic}\label{sec:a computation equal}

Maybe it is a good idea to explain in more details the characteristic $p$ case. Suppose thus $E=\Fq\llparent\pi\rrparent$. There is then an isomorphism of $\Fq$-vector spaces 
\begin{eqnarray*}
(R^{\circ\circ})^d & \iso & \O(Y_{R,R^+})^{\ph=\pi^d} \\
(x_0,\dots,x_{d-1}) & \longmapsto & \sum_{i=0}^{d-1}\sum_{n\in  \Z} x_i^{q^{-n}}\pi^{nd+i}.
\end{eqnarray*}
Consider the formal group $\mathscr{G}_d=\widehat{\mathbb{G}}_a^d$ over $\Fq$ equipped with the  action of $\O_E=\Fq\llbracket\pi\rrbracket$ where the action of $\pi$ is given by 
$$
(X_0,\dots,X_{d-1}) \longmapsto (X_{d-1}^q,X_0,\dots,X_{d-2}).
$$
Then the preceding isomorphism is an isomorphism of $E$-vector spaces 
$$
\mathscr{G}_d (R^\circ) \iso \O(Y_{R,R^+})^{\ph=\pi^d}.
$$
Consider now the power series
$$
f(T)=\sum_{n\geq 0} \frac{T^{q^{n}}}{\pi^n} \in E \llbracket T\rrbracket.
$$
One checks 
$$
f^{-1} (\pi f) \in \O_E \llbracket T\rrbracket,
$$
where $f^{-1}$ is the inverse with respect to composition, 
and satisfies 
$$
f^{-1} (\pi f) \equiv  \pi T+ \a T^q  \text{ mod }T^{q^2}\O_E\llbracket T\rrbracket
$$
with $\a\in 1+\pi\O_E$. Note $\mathscr{G}$ for $\mathscr{G}_1$. 
There is thus a lift $\widetilde{\mathscr{G}}$ of $\mathscr{G}$ over $\O_E$ whose underlying formal group is $\widehat{\mathbb{G}}_a$ and such that the action of $\pi$ is given by
$$
[\pi]_{\widetilde{\mathscr{G}}} = f^{-1}(\pi f)
$$
that is to say
$$
f=\log_{\widetilde{\mathscr{G}}}.
$$
Choose $\e\in R^{\circ\circ}\cap R^\times$ a pseudo-uniformising element. We're seeking for a Weierstrass product expansion of 
$$
t_\e:=\sum_{n\in \Z} \e^{q^{-n}}\pi^n \in \O(Y_{R,R^+})^{\ph=\pi}.
$$
  Define
$$
\eta = \underset{k\drt +\infty}{\lim} [\pi^k]_{\widetilde{\mathscr{G}}} ( \e^{q^{-k}})\in R^{\circ\circ}\llbracket\pi\rrbracket
$$
where the limit is for the $(\varpi_R,\pi)$-adic topology of $\A=R^\circ \llbracket\pi\rrbracket$. Then set 
\begin{eqnarray*}
u_\e &=& \frac{\eta}{\ph^{-1}(\eta)} \\
&=& \frac{[\pi]_{\widetilde{\mathscr{G}}} (\ph^{-1}(\eta))}{\ph^{-1}(\eta)} \in \A.
\end{eqnarray*}
One checks that $u_\e$ is primitive of degree $1$ and moreover $u_\e= \e^{\frac{q-1}{q}} + \pi (1+ \beta )$ with $\beta\in R^{\circ\circ}\llbracket\pi\rrbracket$. We can thus form the convergent Weierstrass product 
$$
\Pi^+ (u_\e)= \prod_{k\geq 0} \frac{\ph^k(u_\e)}{\pi} \in \O (Y_{R,R^+}).
$$
Now set
$$
\Pi^- (u_\e) = \ph^{-1}(\eta) 
$$
which we can think of as being "$\dpt{\prod_{k<0}\ph^k (u_\e)}$" although this product is not convergent. Even if this product has no meaning let us remark that we have partial factorizations
$$
\eta = \ph^{i-1}(\eta). \prod_{i\leq k<0} \ph^{k} (u_\e)
$$ 
for any $i<0$.
We have 
$$
\Pi (u_\e) = \Pi^+(u_\e) \Pi^- (u_\e) \in \O(Y_{R,R^+})^{\ph=\pi} 
$$
and
$$
\Pi(u_\e) = t_\e.
$$
This last equality is a consequence of the Weierstrass product expansion of the rigid analytic function $\log_{\widetilde{\mathscr{G}}}$ as an holomorphic function on $\mathring{\mathbb{B}}^1_E$:
$$
\log_{\widetilde{\mathscr{G}}} = \underset{k\drt +\infty}{\lim} \frac{[\pi^k]_{\widetilde{\mathscr{G}}}}{\pi^k}.
$$
Consider 
$$
R_\e = \A\unpi /u_\e
$$
the "untilt" of $R$ associated to the primitive element $u_\e$ (since we are in equal characteristic, an untilt is nothing else than giving a continuous morphism $E\drt R$, that is to say in fact $R_\e=R$ once we have applied Weierstrass factorization to $u_\e$
but we don't want to fix such an identification). Note
$$
\theta_\e: \A\unpi\drt R_\e
$$
the projection.  Then $\big ( \theta_\e (\ph^{-n}(\eta))\big )_{n\geq 0}$ is a basis of the Tate module  $T_\pi (\widetilde{\mathscr{G}}\hat{\otimes}_{\O_E} R_\e^{\circ})$. 
\\

Let now $I\subset ]0,1[$ be an interval whose extremities are in $p^\Q$ and consider $\DD^*_{(R,R^+),I}$ the corresponding annulus of $\DD^*_{R,R^+}$ via the function $\delta:|Y_{R,R^+}|\drt ]0,1[$ defined after fixing  some $\varpi_R$. Let us fix $\|\varpi_R\|=1/p$. 

\begin{defi}
If $\xi=\sum_{n\geq 0} a_n\pi^n$ is primitive of degree $1$ set $\|\xi\|=\|a_0\|$.
\end{defi}

Let us remark that if $u\in R^\circ\llbracket\pi\rrbracket^\times$ then $\|u\xi\|=\|\xi\|$. Thus, if $\xi=u(\pi-a)$ is the Weierstrass fatorization of $\xi$ then $\|\xi\|=\|a\|$.

Let's come back to the preceding situation. We have $\|u_\e\|=\|\e\|^{\frac{q-1}{q}}$. Then one checks that
\begin{itemize}
\item 
if $A=\{k\in \Z \|\ \|u_{\e^{q^k}}\|\in I$ then  
 $t_\e =  \prod_{k\in A} \ph^k (u_\e)\times u$ with $u\in \O(\DD^*_I)^\times$.
\item For any $n\geq 2$, $O(\DD^*) u_\e + \O(\DD^*) u_{\e^n} =\O(\DD^*)$.
\end{itemize}

One deduces the following: if $(n,p)=1$ then 
$$
\O(\DD^*_I) t_\e + \O(\DD^*_I) t_{\e^n} = \O(\DD^*_I).
$$
We have proved the following.

\begin{prop}
There exists $\e_1,\e_2\in R^{\circ\circ}\cap R^\times$ such that 
$$
\O(\DD^*_I) t_{\e_1}+ \O (\DD^*_I) t_{\e_2}=\O(\DD^*_I).
$$
\end{prop}

We thus have a covering 
$$
\spec ( \O(\DD^*_I)) = \bigcup_{\e\in R^{\circ\circ}\cap R^\times} D (t_\e).
$$
For such an $\e$, there is a morphism
$$
\spec \big ( \O(\DD^*_I)[\tfrac{1}{t_\e}]\big ) \ldrt D^+(t_\e)\subset \text{Proj} (P).
$$
When $\e$ and $I$ vary this gives rise to a {\it $\ph$-invariant} morphism of ind-schemes 
$$
\underset{I}{\limi} \spec\big ( \O(\DD^*_I)\big )\ldrt \text{Proj} (P) = X^{sch}.
$$
Via the morphism or ringed spaces 
$$
\DD^*_I\ldrt \spec \big ( \O (\DD^*_I)\big )
$$
(support of a valuation at the level of topological spaces) this induces the morphism
$$
X^{ad}\ldrt X^{sch}.
$$

\subsubsection{The analog computation in equal characteristic}

Almost all the results of the preceding section extend to the case when $[E:\Qp]<+\infty$. 
The almost here is due to the fact that when $d\geq 2$ then $(R,R^+)\mapsto \O (Y_{R,R^+})^{\ph=\pi^d}$ is not representable by a formal group (or rather its universal cover). But for $d=1$, if we take $\widetilde{\mathscr{G}}$ the Lubin-Tate group law over $\O_E$ with logarithm 
$$
\log_{\widetilde{\mathscr{G}}} = \sum_{n\geq 0} \frac{T^{q^{n}}}{\pi^n}
$$
with reduction $\mathscr{G}$
then 
\begin{eqnarray*}
\mathscr{G} (R^\circ) & \iso & \O(Y_{R,R^+})^{\ph=\pi} \\
\e & \longmapsto & t_\e:=\sum_{n\in \Z} \big [\e^{q^{-n}}\big ]\pi^n.
\end{eqnarray*}
Moreover, for $\e\in R^{\circ\circ}\cap R^\times$, $t_\e$ has a Weierstrass product expansion 
$$
t_\e = \Pi^- (u_\e) .\Pi^+ (u_\e)
$$
with 
$$
u_\e = \frac{\eta}{\ph^{-1}(\eta)}
$$
and 
$$
\eta = \underset{k\drt+\infty}{\lim} [\pi^k]_{\widetilde{\mathscr{G}}} \big ([\e^{q^{-k}}]\big ).
$$
All of this allows us to construct the morphism of ringed spaces 
$$
X^{ad}_{R,R^+}\ldrt X_R^{sch}.
$$

\subsubsection{Intrinsic definition of the schematical curve and GAGA}

We defined $X^{sch}$ as the Proj of the graded algebra $P$. It is thus equipped canonically with a line bundle $\O(1)=\widetilde{P[1]}$. But in fact there is no reason for fixing such a line bundle: {\it the curve exists in itself without a fixed line bundle $\O(1)$ on it}. For example the definition of $\O(1)$ we took depends on the choice of the uniformizing element $\pi$ and another choice may lead to a non-isomorphic $\O(1)$. It will moreover appear clearly in the next section that there are plenty of other choices of other ample line bundles on $X^{sch}$ associated to untilts of $R$, the one of the preceding section being the untilts $R^\sharp$ such that $T_\pi (\widetilde{\mathscr{G}}\hat{\otimes}_{\O_E} R^{\sharp,\circ})=\underline{\O_E}$ (a condition that depends on the choice of $\pi$ since by definition $\widetilde{\mathscr{G}}$ depends on this choice).
\\

Nevertheless we have the following proposition whose proof is not such difficult (see section 6.7 of \cite{Courbe} for the case of a base field).

\begin{prop}
The $\ph$-invariant morphism of ind-schemes 
$$
\underset{I}{\limi} \spec ( \O (Y_I))\ldrt X^{sch}
$$
makes $X^{sch}$ a categorical quotient of the ind-scheme $\underset{I}{\limi} \spec ( \O (Y_I))$ by $\ph^\Z$.
\end{prop}

\begin{coro}
We have to change the definition of the schematical curve $X^{sch}$ to make it intrinsic: $X^{sch}$ is the the categorical quotient of $\underset{I}{\limi} \spec ( \O (Y_I))$ by $\ph^\Z$.
\end{coro}

Of course, to prove such a categorical quotient exists the best way is to take definition \ref{defi:schematical curve} as we did.
\\

We will need the following theorem later.

\begin{theo}[\cite{KedlayaLiuRelative1}]\label{theo:GAGA}
The GAGA functor associated to the morphism of ringed spaces $X^{ad}_{R,R^+}\drt X^{sch}_R$ is an equivalence 
$$
\text{Bun}_{X_R}\iso \text{Bun}_{X^{ad}_{R,R^+}}.
$$
\end{theo}

\subsection{Untilts as Cartier divisors on the schematical curve}

Let $S\in \Perf_{\Fq}$. As we saw any untilt $S^{\sharp}$ of $S$ gives rise to a {\it Cartier divisor }
$$
S^{\sharp}\longhookrightarrow Y_S.
$$
If $S=\spa (R,R^+)$ is affinoid perfectoid then this cartier divisor is defined by the invertible ideal $\O_{Y_S}.\xi\subset \O_{Y_S}$ where $\xi$ is a primitive element associated to $S^{\sharp}$.
\\
This Cartier divisor $D$ gives rise to a  $\ph$-invariant Cartier divisor 
$$
\bigoplus_{k\in \Z} \ph^{k*}D
$$
that is to say a Cartier divisor on $X_{S}^{ad}$. If $S=\spa (R,R^+)$, 
according to GAGA this has to come from a Cartier divisor on $X^{sch}$. Here is its description. 

\subsubsection{The equal characteristic case}

Let $\xi\in R^\circ\llbracket\pi\rrbracket$ be primitive of degree $1$. The associated Cartier divisor on $Y=\DD^*_{R,R^+}$ is given by the line bundle 
$\xi^{-1}\O_Y$ together with its section $1\in \O(Y)$. Let $\L_\xi$ be the associated line bundle on $X^{ad}$. One computes that for any $d\geq 0$
\begin{equation}\label{eq:sections de L xi tordues par pi puissance d}
\GG ( X^{ad},\L_\xi (d))\subset \O(Y)[\tfrac{1}{\xi}]^{\ph=\pi^d}
\end{equation}
and the section of $\L_\xi^{ad}$ defining the Cartier divisor on $X^{ad}$ is given by $1\in \O(Y)[\tfrac{1}{\xi}]^{\ph=Id}$. Up to multiplying $\xi$ by an element of $R^{\circ}\llbracket\pi\rrbracket^\times$ we can suppose $\xi = \pi-a$ with $a\in R^{\circ\circ}\cap R^\times$. Let us form the convergent {\it Weierstrass product }
\begin{eqnarray*}
\Pi^+(\xi) &:=&
\prod_{n\geq 0} \frac{\ph^n(\xi)}{\pi}  \\ &=& \prod_{n\geq 0} \Big ( 1- \frac{a^{q^n}}{\pi}\Big ).
\end{eqnarray*}
It satisfies the functional equation
$$
\ph \big ( \Pi^+(\xi)\big ) = \frac{\xi}{\pi}. \Pi^+(\xi).
$$
Then one checks that via (\ref{eq:sections de L xi tordues par pi puissance d})
$$
\times \Pi^+(\xi):
\GG ( X^{ad},\L_\xi^{ad} (d))  \iso  \O(Y)^{\xi\ph =\pi^{d-1}}.
$$
Define now the following graded module over $P=\bigoplus_{d\geq 0} \O(Y)^{\ph=\pi^d}$
$$
N_\xi = \bigoplus_{d\geq 0} \O(Y)^{\xi\ph=\pi^{d-1}}.
$$
One then checks the following. 
\begin{prop}\label{prop:Cartier via GAGA equal}
The quasi-coherent sheaf $\L_\xi=\widetilde{N_\xi}$ is a line bundle on $X^{sch}=\text{Proj} (P)$. Moreover, equipped with the section $\Pi^+(\xi)\in \GG(X^{sch},\L_\xi)$ this defines a Cartier divisor on $X^{sch}$ that corresponds to the Cartier divisor  on $X^{ad}$ defined by $\xi$ via GAGA.
\end{prop}

\subsubsection{The unequal characteristic case}
\label{sec:untilt as cart unequal}

Let $\xi\in \A=W_{\O_E}(R^\circ)$ be primitive of degree $1$. If $\widetilde{R}=R^\circ/R^{\circ\circ}$ then by definition the reduction $\widetilde{\xi}$ of $\xi$ to $W_{\O_E} (\widetilde{R})$ can be written as $\pi$ times a unit. From this one deduces that up to multiplying $\xi$ by a unit in $\A$ we can suppose that $\widetilde{\xi}=\pi$ that is to say 
$$
\xi \in \pi + W_{\O_E}( R^{\circ\circ}).
$$
We can then form the convergent Weirstrass product 
$$
\Pi^+(\xi):= \prod_{n\geq 0} \frac{\ph^{n}(\xi)}{\pi}\in \O(Y).
$$
We then have the analog of proposition \ref{prop:Cartier via GAGA equal}.

\begin{prop}
The quasi-coherent sheaf $\L_\xi = \big ( \bigoplus_{d\geq 0} \O(Y)^{\xi\ph=\pi^{d-1}} \big )^{\widetilde{\ \ \ \ }}$ is a line bundle on $X^{sch}_R$. It is equipped with the section $\Pi^+(\xi)\in \GG (X^{sch},\L_\xi)$. The couple $(\L_\xi,\Pi^+(\xi))$ defines a Cartier divisor on $X^{sch}_R$ that corresponds via GAGA to the Cartier divisor on $X^{ad}_{R,R^+}$ associated to the untilt of $R$ defined by $\xi$.
\end{prop}

\subsubsection{Formal completion along the Cartier divisor}

Let $R$ be a perfectoid $\Fq$-algebra. Let $R^\sharp$ be an untilt of $R$ and $\xi\in \A$ an associated degree $1$ primitive element.

\begin{defi}
We note $B^+_{dR} ( R^\sharp)$ for the $\xi$-adic completion of $\A\unpi$. We note $B_{dR} ( R^\sharp)=B^+_{dR} ( R^\sharp)[\tfrac{1}{\xi}]$.
\end{defi}

One then has the following. 

\begin{prop}
The formal completion of $X^{sch}_R$ along the Cartier divisor defined by $R^\sharp$ is 
$\spf \,( B^+_{dR} (R^\sharp))$. 
\end{prop}

A usual, the equal characteristic case is much simpler. In this case $R^\sharp$ is given by a morphism $\Fq\llparent\pi\rrparent\drt R$ such that $\pi\mapsto a$ and one can take $\xi=\pi-a$.
Then $B^+_{dR} ( R^\sharp)$ is the $\xi$-adic completion of $R^\circ \llbracket\pi\rrbracket \unpi$. There is thus a {\it canonical section}
$$
\xymatrix@R=4mm{
B^+_{dR} ( R^\sharp)\ar@{->>}[r]^-{\theta} & R^\sharp \ar@{=}[d] \\
& R \ar@{_(->}[lu].
}
$$
This gives a canonical isomorphism 
\begin{eqnarray*}
R\llbracket T\rrbracket &\iso & B^+_{dR} (R^\sharp) \\
T & \longmapsto & \xi.
\end{eqnarray*}
But one has to be a little bit careful: this is an isomorphism of $E$-algebras where $R\llbracket T \rrbracket$ is equipped with the $E$-algebra structure given by
\begin{eqnarray*}
\Fq\llparent \pi \rrparent & \ldrt & R\llbracket T\rrbracket \\
\pi & \longmapsto & T+a.
\end{eqnarray*}

\section{The stack $\Bun_G$}
\label{sec:The stack BunG}

\subsection{Definition and first properties}
\label{sec:defi and prop of Bung}

Let $E$ be either $\Fq\llparent\pi\rrparent$ or a finite degree extension of $\Qp$ with residue field $\Fq$. Let $G$ be a reductive group over $E$.
Consider $\Perf_{\Fq}$ {\it equipped with the pro-étale topology.} 

\begin{defi}
We note $\text{Bun}_G$ the fibred category over $\Perf_{\Fq}$
$$
S\longmapsto \text{Groupoid of}\otimes\text{exact functors from } \text{Rep} (G)\text{ to }\Bun_{X_{S}}.
$$ 
\end{defi}

One has the following technical result that the reader can assume.

\begin{prop}
$\Bun_G$ is a stack on $\Perf_{\Fq}$.
\end{prop}


\begin{Hope}
$\Bun_G$ is a "smooth diamond stack". 
\end{Hope}

{\it There is no precise definition of this notion up to now. That being said, here are a few facts to try to convince the reader.} We treat the $\GL_n$-case.  One can reduce to this case by standard techniques. In fact, if one fixes an embedding $G\subset \GL_n$ then there exists a linear representation $\rho:\GL_n\drt \GL(W)$ and a line $D\subset W$ such that $G$ is the stabilizer of $D$ inside $\GL_n$. Then, $G$-bundles are the same as a vector bundle $\E$ of rank $n$ together with a locally direct factor sub-line bundle of $\rho_*\E$.
\\

\begin{prop}
Let $S\in \Perf_{\Fq}$ and $\E_1,\E_2$ two vector bundles on $X_S$ of the same constant rank. Then the sheaf
$$
T/S\longmapsto \text{Isom} (\E_{1|X_T},\E_{2|X_T})
$$
is a diamond over $S$. When $E=\Fq\llparent \pi\rrparent$ this is representable by a perfectoid space over $S$.
\end{prop}

In fact $\text{Isom} \subset \Hom$ is relatively representable by an open subset. To this this let $u:\E_{1|X_T}\drt \E_{2|X_T}$ be a morphism 
and consider $U\subset X_T$ the biggest open subset on which $u$ is an isomorphism. Note $f:|X_T|\drt |T|$, a  continuous closed application. 
Then $T\setminus f (|X_T|\setminus U)\subset T$ is an open subset such that $T'\drt T$ factorizes through $U$ if and only if $u_{|X_{T'}}$ is an isomorphism.
\\

It thus suffices to prove that $T/S\mapsto \GG ( X_T,(\E_1^\vee \otimes \E_2)_{|X_T})$ is representable. 
This is a consequence of the following.

\begin{prop}
Let $\E$ be a vector bundle on $X_S$ and let $f:(X_S)_{\text{pro-ét}}\ldrt S_{\text{pro-ét}}$. Then $f_*\E$ is representable by a diamond, resp. a perfectoid space when $E=\Fq\llparent\pi\rrparent$. 
\end{prop}

One can suppose $S$ is affinoid perfectoid. Then a theorem of Kedlaya and Liu (\cite{KedlayaLiuRelative1}) says that for $d\gg 0$, $\E(d)$ is generated by its global sections (that is to say $\O(1)$ is ample). Applying this to $\E^\vee$ one finds that there exists $d_1,d_2\in \Z$, $n_1,n_2\in \N$ and a partial resolution 
$$
0\ldrt \E \ldrt \O(d_1)^{n_1} \ldrt \O(d_2)^{n_2}. 
$$
This gives the partial resolution 
$$
0\ldrt f_*\E \ldrt f_*\O(d_1)^{n_1}\ldrt f_*\O (d_2)^{n_2}.
$$
But now, for all $d\in \Z$, $f_*\O(d)$ is representable by a diamond, resp. a perfectoid space if $E=\Fq\llparent \pi\rrparent$. For $E|\Qp$ this is a consequence of the theory of Banach-Colmez spaces (\cite{Colmez2}). For $E=\Fq\llparent\pi\rrparent$ this is $0$ if $d<0$, $\underline{E}$ is $d=0$ and represented by a $d$-dimensional perfectoid open ball over $S$ if $d>0$. More precisely, for the case $E=\Fq\llparent \pi\rrparent$:
\begin{itemize}
\item if $S=\spa (R,R^+)$ then $\underline{E}_{|S} =\cup_{n\geq 0}\spa \big ( \mathscr{C} (\pi^{-n}\O_E,R),\mathscr{C}(\pi^{-n}\O_E,R^+)\big )$.
\item if $d>0$, with the notations of section \ref{sec:a computation equal}, $f_*\O(d)= (\underset{\text{Frob}}{\limp} \mathscr{G}_d)\times S$ where $\underset{\text{Frob}}{\limp} \mathscr{G}_d =\underset{\times \pi}{\limp} \mathscr{G}_d$ is the universal covering of the formal group $\mathscr{G}_d$ (a formal $E$-vector space).
\end{itemize}
Thus, $f_*\E$ is the kernel of a morphism of group diamonds, resp. perfectoid spaces, and is thus representable.

\begin{coro}\label{coro:diagonale diamond}
The diagonal of $\text{Bun}_G$ is representable by a diamond, resp. a perfectoid space if $E=\Fq\llparent \pi\rrparent$.
\end{coro}

Now comes the problem of the existence of a {\it "smooth" presentation of $\text{Bun}_G$}. First, let us remark that $\spa (\Fq \llparent T^{1/p^\infty}\rrparent)\drt \spa (\Fq)$ is a an épimorphism for the analytic topology. In fact, is $(R,R^+)$ is affinoid perfectoid then to each choice of a pseudo-uniformizing element $\varpi_R$ there is associated a morphism $(\Fq\llparent T^{1/p^\infty}\rrparent,\Fq\llbracket T^{1/p^\infty}\rrbracket)\drt (R,R^+)$ via $T\mapsto \varpi_R$. The projection $\Bun_G\times \spa (\Fq \llparent T^{1/p^\infty}\rrparent )\ldrt \Bun_G$ is thus an epi and let's declare it to be smooth. We are thus reduced to find a "smooth" presentation of $\Bun_G\times \spa (F)$ with $F=\Fq\llparent T^{1/p^\infty}\rrparent$. Now, {\it we take inspiration from the theory of  Quot schemes}. 
\\

Let's take the following definition. Let $S\in \Perf_{\Fq}$. Note $f:(X_S)_{\text{pro-ét}}\drt S_{\text{pro-ét}}$. For $d\geq 0$ we define 
$$
\mathbb{V}_{d/S} = f_* \O_{X_S}(d).
$$
This is a Banach-Colmez $E$-vector space. More precisely, if $\mathbb{B}$ is the following sheaf of $E$-algebras on $\Perf_{\Fq}$,
$$
\mathbb{B}(S) = \O(Y_S),
$$
then
$$
\mathbb{V}_d= \mathbb{B}^{\ph=\pi^d}.
$$
This has  a crystalline description when $E=\Qp$ (and in fact more generally when $E|\Qp$ using $\pi$-divided powers). In fact, for $R$ a perfectoid $\Fq$-algebra consider 
\begin{eqnarray*}
B^+_{cris} (R^\circ/\varpi_R) &=& H^0_{cris} \big ( \spec ( R^\circ/\varpi_R)/\spec (\Zp),\O\big )\unp \\
&=& \widehat{W (R^\circ)\big [ \tfrac{[\varpi_R^n]}{n!}\big ]} \unp.
\end{eqnarray*}
There is an inclusion (see sec.\ref{sec:F isoc and vb} for more details)
$$
B^+_{cris} (R^\circ/\varpi_R) \subset \O(Y_{R,R^+})
$$
that induces an isomorphism
$$
B^+_{cris} (R^\circ/\varpi_R)^{\ph=\pi^d}\iso \mathbb{B} (R)^{\ph=\pi^d}.
$$

 This is a diamond and even a perfectoid space when $E=\Fq\llparent\pi\rrparent$. We have $\mathbb{V}_0=\underline{E}$. This forms a graded $E$-algebra 
$$
\bigoplus_{d\geq 0} \mathbb{V}_{d/S}
$$
in the category of diamonds over $S$ (resp. perfectoid spaces over $S$ when $E=\Fq\llparent\pi\rrparent$).

\begin{exem}
If $E=\Fq\llparent\pi\rrparent$ and $d>0$ consider the perfect formal scheme $\mathcal{V}_d=\spf ( \Fq\llbracket x_0^{1/p^\infty},\dots,x_{d-1}^{1/p^\infty}\rrbracket)$. This represents the universal cover of the formal group $\mathscr{G}_d$ of section (\ref{sec:a computation equal}), a formal $E$-vector space. There are morphisms 
$\mathcal{V}_{d_1}\times \mathcal{V}_{d_2}\drt \mathcal{V}_{d_1+d_2}$ obtained by expanding the product 
$$
\Big ( \sum_{i=0}^{d_1-1}\sum_{n\in \Z} x_i^{q^{-n}} \pi^{nd_1+i}\Big ).\Big ( \sum_{j=0}^{d_2-1}\sum_{m\in \Z} y_j^{q^{-m}} \pi^{md_2+j}\Big ).
$$
For example, $\mathcal{V}_1\times \mathcal{V}_1\drt \mathcal{V}_2$ is given by $$(x,y)\longmapsto \Big ( \sum_{n\in \Z} x^{q^{-n}}y^{q^n},
\sum_{n\in\Z} x^{q^{-n}} y^{q^{n-1}}\Big )
.$$
 Then, taking the the generic fiber over $S$, 
$$
\bigoplus_{d>0} \mathbb{V}_{d/S} = \bigoplus_{d>0} \mathcal{V}_d\times S
$$
where $\mathcal{V}_d\times S$ is a $d$-dimensional perfectoid open ball.
\end{exem}

Any closed subset of a diamond is a diamond. Any "locally Zariski closed" subset of a perfectoid space is a perfectoid space. The following definition thus makes sense. It is inspired by the following: an affine algebraic variety is defined by the vanishing/non-vanishing of a finite set of polynomials in a finite dimension vector space...we thus can do the same with a finite dimensional Banach space in Colmez sense (\cite{Colmez2}).

\begin{defi}
An algebraic diamond, resp. perfectoid space when $E=\Fq\llparent\pi\rrparent$, over $S$ is a subsheaf of $\mathbb{V}_{d/S}^n$ for some $d,n\geq 1$ defined by the vanishing and non-vanishing of a finite collection of polynomials in $E[X_1,\dots,X_n]$.
\end{defi}

\begin{prop}\label{prop:Pres est un diamant}
For $\mathcal{D}=(d_1,d_2,m_1,m_2,n)$ with $d_1,d_2\in \Z$ satisfying $d_2<d_1$, $m_1,m_2\in \N_{\geq 1}$ and $n\leq m_1$
the functor 
\begin{eqnarray*}
\text{Pres}_{\mathcal{D}}:\Perf_F & \ldrt & \text{Sets} \\
S & \longmapsto & \big \{ u\in \Hom \big ( \O_{X_S}(d_2)^{m_2},\O_{X_S} (d_1)^{m_1} \ |\ \text{coker } u\text{ is locally free of rank }n\big \}
\end{eqnarray*}
is representable by an algebraic diamond, resp. perfectoid space, over $F$.
\end{prop}

In fact,
the sheaf  
$$
S\longmapsto \Hom \big ( \O_{X_S} (d_2)^{m_2},\O_{X_S}(d_1)^{m_1}).
$$
 is isomorphic to $\mathbb{V}_{d_1-d_2}^{m_1m_2}$ and we see it as $M_{m_1,m_2} ( \mathbb{V}_{d_1-d_2})$, size $m_1\times m_2$ matrices with coefficients in $\mathbb{V}_{d_1-d_2}$. Suppose now that $n\leq m_1$. For $u\in M_{m_1,m_2}$, $u=(X_{ij})_{1\leq i\leq m_1,1\leq j\leq m_2}$, look at the polynomials $(P_\a)_\a$, resp. $(Q_\beta)_\beta$,
 in $E[X_{ij}]_{1\leq i\leq m_1,1\leq j\leq m_2}$ that are the coefficients of $\wedge^{n+1}u$, resp. $\wedge^n u$. Then the vanishing of all the $(P_\a)_\a$ and the non-vanishing of one of the $(Q_\beta)_\beta$ defines an algebraic diamond in $\mathbb{V}_{d_1-d_2}^{m_1m_2}$. This 
represents $\text{Perf}_{\mathcal{D}}$.
\\

There is an evident morphism
$$
pres_{\mathcal{D}}:\text{Pres}_{\mathcal{D}}\ldrt \text{Bun}_G\underset{\spa (\Fq)}{\times} \spa ( F).
$$
One can then check the following.

\begin{prop}
We can write $\Bun_G$ as a countable union of quasi-compact open subsets 
$$
\Bun_G= \bigcup_{n\geq 1} U_n 
$$
such that for each $n$ there exists a datum $\mathcal{D}_n$ as in proposition \ref{prop:Pres est un diamant} satisfying:
\begin{enumerate}
\item $U_n\times \spa (F)\subset \text{Im} (pres_{\mathcal{D}_n})$
\item the geometric fibers of $pres_{\mathcal{D}_n}$ restricted to $pres_{\mathcal{D}_n}^{-1}(U_n\times \spa (F))$ satisfy the following:
\begin{itemize}
\item if $E=\Fq\llparent\pi\rrparent$ they are isomorphic to perfectoid open balls of the same dimension (independently of the geometric point of $U_n$) 
\item if $E|\Qp$ there exists $d\geq 1$ and $h\in \N$ such that for each geometric fiber 
$Z$ there exists a pro-étale Galois cover $\widetilde{Z}\drt Z$ with Galois group $E^h$ and $\widetilde{Z}$ is isomorphic to a perfectoid open ball of dimension $d$.
\end{itemize}
\end{enumerate}
\end{prop}

In any sense we can think about it, the morphisms 
$$
\big (pres_{\mathcal{D}_n}\big )_{|pres_{\mathcal{D}_n}^{-1}(U_n)} 
$$
are smooth and surjective. This gives us a a "presentation" of $\text{Bun}_G$.

%
%

%
%

\subsection{Points of the stack}

Fix an algebraic closure $\Fqb$ of $\Fq$ and note $L=\widehat{E^{un}}$ the completion of the corresponding maximal unramified extension of $E$. Let $\s$ be its Frobenius. Recall (Kottwitz \cite{Ko2}) that  
$$
B(G):= G (L)/\s\text{-conjugation}.
$$
We note
$$
\phmod_L =\{(D,\ph)\}
$$
for the associated category of isocrystals as classified by Dieudonné and Manin where $D$ is a finite dimensional $L$-vector space and $\ph$ a $\s$-linear automorphism of $D$.
To each $b\in G(L)$ there is associated an isocrystal with a $G$-structure
\begin{eqnarray*}
\mathcal{F}_b: \text{Rep} (G)&\ldrt & \phmod_L \\
(V,\rho) & \longmapsto & (V\otimes_E L,\rho (b)\s).
\end{eqnarray*}
And this defines a bijection
\begin{eqnarray*}
G\text{-isocrystals}/\sim & \iso & B(G) \\
\ [b] & \longmapsto & [\mathcal{F}_b].
\end{eqnarray*}

\subsubsection{Classification of vector bundles}

Fix now $F|\Fqb$ perfectoid and consider the corresponding curve $X^{ad}$ or $X^{sch}$. There is a functor 
\begin{eqnarray*}
\phmod_L & \ldrt & \text{Bun}_{X^{ad}} \\
(D,\ph) & \longmapsto & \E(D,\ph)^{ad}
\end{eqnarray*}
defined in the following way. Since $F|\Fqb$, the space $Y$ sits over $\spa (L)$. Then the geometric realization of $\E(D,\ph)^{ad}$ is given by 
$$
Y\underset{\ph}{\times} D \ldrt Y/\ph^{\Z}=X^{ad}.
$$
That is to say the vector bundle $\E(D,\ph)^{ad}$ becomes trivial with fiber $D$ when pulled back to $Y$ via the covering $Y\drt X^{ad}$ and {\it the corresponding automorphy factor is given by $\ph$ acting on $D$.} Via GAGA 
$$
\text{Bun}_{X^{sch}}\iso \text{Bun}_{X^{ad}}
$$
the vector bundle $\E(D,\ph)^{ad}$ is the analytification of the algebraic vector bundle $\E(D,\ph)$ associated to the graded algebra
$$
\bigoplus_{d\geq 0} \big (D\otimes_L \O(Y)\big )^{\ph\otimes \ph=\pi^d}
$$
over $P=\oplus_{d\geq 0} \O(Y)^{\ph=\pi^d}$, $X^{sch}=\text{Proj}(P)$.
\\
We then have the following theorem.

\begin{theo}[\cite{Courbe}]
If  $F$ is algebraically closed
 the functor $\E(-):\phmod_L\drt \text{Bun}_X$ is essentially surjective.
\end{theo}

There is a more concrete way to state this theorem via Dieudonné-Manin classification. In fact we have defined for each $d\in \Z$ a line bundle 
$\O(d)$ on $X$. One has $\O(d)=\E ( L,\pi^{-d}\s)$. For an integer $h\geq 1$ note $E_h$ the degree $h$ unramified extension of $E$. One has
$$
X_E\otimes_E E_h = X_{E_h} 
$$
a formula that reflects the fact that replacing $E$ by $E_h$ does not change Fontaine's ring $\A$ but replaces $\ph$ by $\ph^h$. The cyclic Galois covering 
$$
X_{E_h}^{ad}\ldrt X_E^{ad}
$$
is then identified with the {\it unfolding covering }
$$
Y/\ph^{h\Z}\ldrt Y/\ph^\Z.
$$
Let us note $X_h:=X_{E_h}$, $X_1=X$, and $\pi_h:X_h\drt X$. We thus have a canonical $\widehat{\Z}$-Galois cover
$$
(X_h)_{h\geq 1}\ldrt X.
$$ 
For $\l\in \Q,\l=\frac{d}{h}$ with $(d,h)=1$ define 
$$
\O(\l) := \pi_{h*} \O_{X_h} (d).
$$
One has $\O(\l)\simeq \E(D,\ph)$ where $(D,\ph)$ is simple with Dieudonné-Manin slope $-\l$.  The preceding theorem can then be restated as saying that, {\it when $F$ is algebraically closed, for any $\E\in \Bun_X$ there exists a finite collection of slopes $(\l_i)_{i\in I}$ such that  }
$$
\E\simeq \bigoplus_{i\in I} \O(\l_i).
$$
One can go further in restating this theorem. In fact, the curve $X^{sch}$ is complete in the sense that for any $f\in E(X^{sch})^\times$, $\deg ( \div f)=0$. As a consequence there is a good degree function on vector bundles on $X^{sch}$ and  {\it Harder-Narasimhan filtrations on vector bundles}. For each $\l\in \Q$, $\O(\l)$ is stable of slope $\l$. Then the preceding theorem can be restated in the following way.

\begin{theo}
If $F$ is algebraically closed:
\begin{enumerate}
\item For any $\l\in\Q$, any slope $\l$ semi-stable vector bundle is isomorphic to a finite direct sum of $\O(\l)$.
\item The Harder-Narasimhan filtration of a vector bundle on $X^{sch}$ is split.
\end{enumerate}
\end{theo}

Let us remark that the second point (2) is still true for any $F$ not necessarily algebraically closed.

We thus have a bijection
\begin{eqnarray*}
\big \{  \l_1\geq \dots\geq \l_n\ |\ n\in \N, \l_i\in \Q\big \}&\iso & \text{Bun}_X /\sim \\
(\l_1,\dots,\l_n) & \longmapsto & \Big [ \bigoplus_{i=1}^n \O(\l_i)\Big ].
\end{eqnarray*}

\subsubsection{Classification of $G$-bundles}

Let $b\in G(L)$. The composite 
$$
\text{Rep} (G) \xrig{\ \mathcal{F}_b\ } \phmod_L \xrig{\ \E(-)\ } \Bun_X
$$
defines a bundle with a $G$-structure on $X$, that is to say a $G$-torsor on the scheme $X^{sch}$ locally trivial for the étale topology
$$
\E_b.
$$
We then have the following generalization of the classification of vector bundles theorem.

\begin{theo} 
\label{theo:classification des G torseurs}
If $E|\Qp$ and $F$ is algebraically closed 
there is a bijection of pointed sets
\begin{eqnarray*}
B(G) & \iso & H^1_{\et} ( X^{sch},G) \\
\ [b] & \longmapsto & \big [ \E_b\big ].
\end{eqnarray*}
\end{theo}

Of course this has to be true when $E=\Fq\llparent \pi\rrparent$.

\begin{conj}
Theorem \ref{theo:classification des G torseurs} still holds when $E=\Fq\llparent\pi\rrparent$. 
\end{conj}

{\it We will admit this conjecture and do as if it where true until the end of this paper.}
\\

When $G$ is quasi-split, 
this bijection has Harder-Narasimhan features as before. For this fix a triple $A\subset T\subset B$ where $A$ is a maximal split torus, $T$ is a maximal torus and $B$ a Borel subgroup. Kottwitz has defined a slope application
\begin{eqnarray*}
B(G) & \ldrt & \Big [\Hom (\DD, G_{E^{un}})/\text{conj.}\Big ]^{\s=Id}= X_*(A)_\Q^+ \\
\ [b] & \longmapsto & [\nu_b]
\end{eqnarray*}
There is a notion of semi-stability for $G$-torsors on the curve $X^{sch}$
analogous to the usual one for "classical curves" in terms of reduction to standard parabolic subgroups of $G$. By definition, $b$ is basic if $\nu_b$ is central. We then have the following equivalence 
$$ 
b
\text{ is basic }  \lssi \E_b\text{ is semi-stable}.
$$
Moreover, if $w$ is the maximal length element in the Weyl group of $T$, so that $B^w$ is opposite to $B$ then 
$$
[w.(-\nu_b)]\in X_*(A)_\Q^+
$$
is a generalized Harder-Narasimhan polygon of $\E_b$. Here the minus sign is due to the fact that under the correspondence $(D,\ph)\mapsto \E(D,\ph)$ Dieudonné-Manin slopes are the opposite of the Harder-Narasimhan slopes.

\subsubsection{Points of the stack}

Any $b\in G(L)$ defines a morphism 
$$
x_b:\spa (\Fqb) \ldrt \Bun_G.
$$
We set
$$
\Bun_{G,\Fqb} := \Bun_G\otimes_{\Fq} \Fqb.
$$
Define now the following
$$
\big | \Bun_{G,\Fqb}\big | = \Big ( \coprod_{F} \Bun_G (F)\Big )/\sim
$$
where $F$ goes through all perfectoid fields extensions of $\Fqb$. By definition, for $x_1\in \Bun_G(F_1)$ and $x_2\in \Bun_G (F_2)$, $x_1\sim x_2$ if there exists $F_3$ and embeddings $\a:F_1\hookrightarrow F_3$, $\b:F_2\hookrightarrow F_3$ such that $\a_*x\simeq \b_*y$.
\\

As a consequence of the preceding classification theorem we obtain then that there is a canonical bijection
$$
B(G) \iso \big |\Bun_{G,\Fqb}\big |.
$$
We will now equip $|\Bun_{G,\Fqb}\big |$ and thus $B(G)$ with the topology whose open subsets are the $|\mathscr{U}|$ with $\mathscr{U}\subset \Bun_{G,\Fqb}$ is an open substack. One has to be careful this is different from the quotient topology given by the formula $B(G)=G(L)/\s\text{-conj.}$ 
(this one is the discrete topology).

\subsection{Connected components}

Kottwitz has defined a map 
$$
\kappa: B(G)\ldrt \pi_1(G)_\Gamma
$$
where $\Gamma =\Gal (\overline{E}|E)$ and $\pi_1(G)$ is Borovoi's fundamental group. When $G=\GL_n$ this gives the endpoint of the Newton polygon that is to say the opposite of the degree of the corresponding vector bundle. In general this can be interpreted as an $G$-equivariant first Chern class of a $G$-torsor (\cite{Gtorseurs} sec.8). The following theorem is easy when $G_{der}$ is simply connected (this is reduced to the torus case via the projection to $G/G_{der}$) but much more subtle in general.

\begin{theo}
Kottwitz map $\kappa$ is locally constant on $\text{Bun}_{G,\Fqb}$.
\end{theo}

We thus have a decomposition in open/closed substacks 
$$
\Bun_{G,\Fqb} = \coprod_{\a\in \pi_1(G)_\Gamma} \Bun_{G,\Fqb}^\a.
$$
The following conjecture is natural by analogy with the "classical case". We won't need it to state our conjecture. 

\begin{conj}
For each $\a\in \pi_1(G)_\Gamma$, $\Bun_{G,\Fqb}^\a$ is connected. 
\end{conj}

\subsection{Harder-Narasimhan stratification of  $\Bun_G$}

{\it We suppose from now on that $G$ is quasi-split.}
\\

Using a result of Kedlaya and Liu (who treated the $\GL_n$-case, see \cite{KedlayaLiuRelative1}) we can prove the following. We equip $X_*(A)_\Q^+$ with the order given by $\nu_1\leq \nu_2$ if $\nu_2-\nu_1$ is a positive $\Q$-linear combination of coroots.

\begin{theo}
The Harder-Narasimhan map $HN:B(G)\drt X_*(A)_\Q^+$ that sends $[b]$ to $[w.(-\nu_b)]$ is semi-continuous in the sense that for all $\nu\in X_*(A)_\Q^+$ the set of $[b]\in B(G)$ such that $HN ([b])\leq \nu$ is closed. 
\end{theo}

We will only use the following corollary to state our conjecture.

\begin{coro}
The semi-stable locus $\Bun_G^{ss}$ is an open substack.
\end{coro}

\subsection{The automorphism group of $\E_b$}

Fix $b\in G(L)$ and consider $J_b$ the $\s$-centralizer of $b$. This is an inner form of a Levi subgroup of $G$, the centralizer of the slope morphism $\nu_b$. In particular, if $b$ is basic this is an inner form of $G$.

\begin{defi}
We note $\widetilde{J}_b$ the pro-étale sheaf on $\Perf_{\Fqb}$ defined by 
$$
\widetilde{J}_b (S) = \text{Aut} ( \E_{b|X_S}).
$$
\end{defi}

According to \ref{coro:diagonale diamond} this is a diamond group (to be more precise when restricted to $\Perf_F$ for any perfectoid field $F$). 
We have 
$$
J_b (E) = \Aut ( \mathcal{F}_b)
$$
and thus an inclusion
$$
\underline{J_b(E)}\subset \widetilde{J}_b.
$$
We then have the following proposition.

\begin{prop}
If $b$ is basic then $\underline{J_b(E)}= \widetilde{J}_b$.
\end{prop}

\begin{exem}
The automorphisms of the trivial $G$-bundle $\E_1$ is given by $\underline{G(E)}$. More precisely, for any $S\in \Perf_{\Fq}$ the automorphisms of the trivial $G$-bundle on $X_S$ is the group of continuous applications $\mathscr{C}(|S|, G(E))$.
\end{exem}

\begin{rema}
Contrary to the "classical situation" (the stack of $G$-bundles on an algebraic curve over a field $k$ with $G/k$) we thus see that {\it the automorphisms of the trivial $G$-bundle is $G(E)$ and not the algebraic group $G$ !}
\end{rema}

But any automorphism of $\E_b$ induces an automorphism of its canonical semi-stable reduction and we thus always have a morphism 
$$
\widetilde{J}_b\ldrt \underline{J_b(E)}.
$$
At the end we have the following structure result for $\widetilde{J}_b$.

\begin{prop}
We have $\pi_0 (\widetilde{J}_b)=J_b(E)$ and $\widetilde{J}_b= \widetilde{J}_b^{0}\rtimes \underline{J_b(E)}$. Moreover the neutral connected component $\widetilde{J}_b^0$ is a unipotent group diamond, resp. 
perfectoid space when $E=\Fq\llparent \pi\rrparent$.
\end{prop}

\begin{exem}\label{ex:auto ordi curve}
Let us treat a simple example. Take $G=\GL_2$ and $b=\text{diag} (1,\pi^{-1})$. We then have $\E_b=\O\oplus \O(1)$. Then
$$
\widetilde{J}_b= \left ( \begin{matrix}
\underline{E}^\times & \mathbb{B}^{\ph=\pi} \\
0 & \underline{E}^\times
\end{matrix}\right ).
$$
As we can see, $\widetilde{J}_b^0= \mathbb{B}^{\ph=\pi}$, a Banach-Colmez space.
\end{exem}

 In general $\widetilde{J}_b^0$ has a filtration whose graded pieces are Banach-Colmez spaces.

\subsection{Uniformization of HN strata}

\subsubsection{The basic case}

We first treat the basic case since it is simpler and this is the only case we need to state our conjecture.
Recall (Kottwitz) that the restriction of $\kappa$ induces a bijection
\begin{equation}
\label{eq:kappa restreint au basic}
\kappa: B(G)_{basic}\iso \pi_1(G)_\Gamma.
\end{equation}
Geometrically this is translated in the following way.

\begin{prop}
For each $\a\in \pi_1(G)_\Gamma$ the basic locus in the component cut out by $\a$, the open subset 
$$
\big | \Bun_{G,\Fqb}^{\a,ss}\big |,
$$
is only one point determined by (\ref{eq:kappa restreint au basic}).
\end{prop}

Fix $b\in G(L)$ with $\kappa ([b])=\a$. We then have a morphism
\begin{equation}
\label{eq:morphisme unif basic}
x_b: \big [ \spa (\Fqb)/\underline{J_b(E)}\big ]\ldrt \Bun_{G,\Fqb}^{\a,ss}
\end{equation}
defined by $\E_b$. Here the left hand side is the classifying stack of $J_b(E)$-pro-étale torsors.
As a consequence of results of Kedlaya and Liu (the $\GL_n$-case for $b=1$, see \cite{KedlayaLiuRelative1}) one can prove the following.

\begin{theo}\label{theo:unif basic loc}
The morphism (\ref{eq:morphisme unif basic}) is an isomorphism
$$
\big [ \spa (\Fqb)/\underline{J_b(E)}\big ]\iso \Bun_{G,\Fqb}^{\a,ss}.
$$
\end{theo}

This theorems means simply that if $S\in \Perf_{\Fqb}$ and $\E$ is a $G$-bundle on $X_S$ such that for each geometric point $\bar{s}\drt S$, $\E_{|X_{k(\bar{s})}}\simeq \E_b$, then there exists a pro-étale covering $T\drt S$ such that $\E_{|X_T}\simeq \E_b$. The associated $\underline{J_b(E)}$-torsor over $S$ is then the one of isomorphisms between $\E_b$ and $\E$.

\subsubsection{The general case}

We won't use this general case in the statement of the conjecture and the reader can skip it.
\\
According to Kottwitz, the map 
$$
B(G)\xrig{\ (\kappa,\nu)\ } \pi_1(G)_\Gamma \times X_*(A)_\Q^+
$$
is injective. Geometrically this is translated in the following way.

\begin{prop}
Fix $\a\in \pi_1(G)_\Gamma$. Then for each $\nu\in X_*(A)_\Q^+$ the stratum $\big |\Bun_{G,\Fqb}^{\a,HN=\nu}\big |$ is either empty or one point.  
\end{prop}

Suppose such a stratum is non-empty and fix $b$ such that $\kappa (b)=\a$ and $w.(-\nu_b)=\nu$. Then theorem \ref{theo:unif basic loc} extends in this context:
$$
\big [\spa (\Fqb)/\widetilde{J}_b\big ]\iso \Bun_{G,\Fqb}^{\a,HN=\nu}
$$
where the right hand side is the stack classifying $G$-bundles that are geometrically fiberwise isomorphic to $\E_b$.
\\
In whatever sense one can imagine, one has 
$$
\dim\, \widetilde{J}_b = <2\rho,\mu>
$$
where $\rho$ is the half sum of the positive weights of $T$ in $\Lie\, B$. We thus have 
$$
\dim\, \Bun_{G,\Fqb}^{\a,HN=\nu} = -<2\rho,\mu>
$$
and we conclude that:
\begin{itemize}
\item {\it The basic locus is zero dimensional}
\item The dimension of HN strata goes to $-\infty$ when we go deeper in the Weyl chambers. 
\end{itemize}

\section{The $B_{dR}$-affine Grassmanian and the Hecke stack}
\label{sec:The BdR affine Grassmanian and the Hecke stack}

\subsection{The $B_{dR}$-affine Grassmanian: definition}

When $E=\Fq\llparent\pi\rrparent$ we adopt the convention $\spa (E)^\diamond=\spa (E)$.
The $B_{dR}$-affine Grassmanian is by definition a pro-étale sheaf 
$$
\xymatrix@R=5mm{\text{Gr}\ar[d] \\ \spa (E)^\diamond}
$$
that is to say a pro-étale sheaf on $\text{Perf}_E$ 
defined in the following way. 

\begin{defi}
Let $(R,R^+)$ be affinoid perfectoid over $E$. Then 
$$
\text{Gr}(R,R^+) = \{(\mathscr{T}, \xi )\}/\sim
$$
where $\mathscr{T}$ is a $G$-torsor on $\spec ( B^+_{dR} (R))$ and $\xi$ a trivialization of $\mathscr{T}\otimes_{B^+_{dR}(R)} B_{dR} (R)$.
\end{defi}

\begin{rema}\label{rem:surconv}
By definition $\text{Gr} (R,R^+)$ does only depend on $R$ and not on the choice of $R^+$. The same phenomenon occurs for $\text{Bun}_G (R,R^+)$, and all the sheaves and stacks showing up in this text like $\spa (E)^\diamond$. Geometrically this tells us that {\it all the geometric objects we consider are partially proper}, that is to say overconvergent in the language of Tate rigid geometry, resp. without boundary in Berkovich language.
\end{rema}

\begin{rema}
The reader may be intrigued about the fact that a "$\spec$" occurs in the definition of $\text{Gr}$ and not a "$\spa$". This is not surprising anymore if one thinks about the equivalence for any affinoid perfectoid $(R,R^+)$ (Kedlaya-Liu)
$$
\text{vector bundles on }\spa (R,R^+) \simeq \text{vector bundles on } \spec (R)
$$
and thus {\it $G$-bundles in the Tannakian sense on $\spa (R,R^+)$ are the same as $G$-torsors locally trivial for the étale topology on $\spec (R)$}.
\end{rema}

One checks that any $G$-bundle on $\spa (R,R^+)$ becomes trivial after an étale covering of $\spa (R,R^+)$. From this one deduces that $\Gr$ is the étale sheaf associated to the presheaf
$$
(R,R^+)\longmapsto G( B_{dR}(R))/G(B^+_{dR}(R)).
$$

\subsection{The Beauville-Laszlo uniformization morphism}\label{sec:BL}

Let $(R,R^+)$ be affinoid perfectoid over $\Fq$ and consider an untilt $R^\sharp$ over $E$ of $R$. This defines a Cartier divisor on the relative schematical curve $X_{R}^{sch}$. The formal completion along this divisor is $\spf ( B^+_{dR}(R^\sharp))$. {\it Applying Beauville-Laszlo gluing} (\cite{BeauvilleLaszlo2}) we thus obtain the following.

\begin{prop}
As a sheaf over $\spa (E)^\diamond$ 
$$
\text{Gr}(R,R^+)=\{(R^\sharp,\E,\xi) \}/\sim 
$$
where 
\begin{itemize}
\item $R^\sharp$ is an untilt of $R$
\item $\E$ is a $G$-bundle on $X_{R}^{sch}$
\item if $D$ is the Cartier divisor on $X_R^{sch}$ defined by $R^\sharp$ then $\xi$ is a trivialization of $\E_{|X_{R}^{sch}\setminus D}$.
\end{itemize}
\end{prop}

As corollary there is a Beauville-Laszlo morphism
$$
\mathcal{BL}: \text{Gr}\ldrt \Bun_G\times_{\spa (\Fq)} \spa (E)^\diamond.
$$
Recall the following that is an analog of a result of Drinfeld-Simpson (\cite{DrinfeldSimpson}).

\begin{prop}[\cite{Gtorseurs}]
For a perfectoid algebraically closed field $F|\Fq$ and a closed point $\infty\in |X_F^{sch}|$, any $G$-bundle on $X_F^{sch}$ becomes trivial on $X_F^{sch}\setminus \{\infty\}$.
\end{prop}

\begin{rema}
Drinfeld and Simpson add the assumption that $G$ is semi-simple. This is not necessary in our case since for such an $F$, $\text{Pic}^0(X_F)=0$.
\end{rema}

Here is the translation of the preceding proposition in geometric terms.

\begin{prop}
The Beauville-Laszlo morphism $\mathcal{BL}$ is surjective at the level of geometric points. 
\end{prop}

Of course we conjecture the following. 

\begin{conj}
The morphism $\mathcal{BL}$ is "surjective locally for the smooth topology". More precisely, 
for $S\in \Perf_{\Fq}$ and an element of $(\text{Bun}_G\times \spa (E)^\diamond)(S)$ there exists a "smooth" surjective cover $\widetilde{S}\drt S$ such that when restricted to $\widetilde{S}$ this element comes from en element of $\text{Gr}(\widetilde{S})$.
\end{conj}

Following Drinfeld-Simpson, this would follow from the following.

\begin{conj}
For $S\in \text{Perf}_{\Fq}$ and $\E$ a $G$-bundle on $X_S$ the moduli space or reductions of $\E$ to $B$ is a "smooth" diamond over $S$, resp. perfectoid space when $E=\Fq\llparent\pi\rrparent$. 
\end{conj}

This should follow from Quot-diamond techniques as in section \ref{sec:defi and prop of Bung}.

\subsection{Geometric structure on the $B_{dR}$-affine Grassmanian}
\subsubsection{The equal characteristic case}

We begin with {\it the case $E=\Fq\llparent\pi\rrparent$ since in this case, as the reader should have understood now, diamonds should not appear and we should be able to work only with perfectoid spaces.} 
\\

Recall (sec.\ref{sec:untilt as cart unequal}) that in this case, for $R$ a perfectoid $E$-algebra, $B^+_{dR}(R)= R\llbracket T\rrbracket$ where the structural morphism  $E\drt R\llbracket T\rrbracket$ is given 
by 
$$
\pi\longmapsto T+a
$$
where $\pi\mapsto a \in R^{\circ\circ}\cap R^\times$ via $E\drt R$ and 
$$
T+a = a(1+\tfrac{T}{a}) \in R\llbracket T\rrbracket^\times
$$

{\it Suppose first that $G$ is unramified} and note $\underline{G}$ its reductive model over $\O_E$.
\\
 Note $\text{Nilp}_{\O_E}$ for the category of $\O_E$-algebras on which $\pi$ is nilpotent. We define a functor 
$$
\mathcal{F}:\text{Nilp}_{\O_E}\ldrt \text{Sets}
$$
 Let $R\in \text{Nilp}_{\O_E}$ with $\pi\mapsto a\in R$. Equip $R\llbracket T\rrbracket$ with the $\Fq\llbracket \pi\rrbracket$-algebra structure given by $\pi\mapsto T+a$.
Since $a$ is nilpotent, $R\llparent T\rrparent$ is an $E$-algebra.  
  We then set
 $$
 \mathcal{F} (R) = G \big ( R\llparent T\rrparent \big )/ \underline{G} (R\llbracket T\rrbracket).
 $$
Now let us note $\widetilde{\mathcal{F}}$ for the étale sheaf associated to $\mathcal{F}$. One can then check the following.

\begin{prop}
\begin{enumerate}
\item 
The étale sheaf $\widetilde{\mathcal{F}}$ is represented by an ind $\pi$-adic $\O_E$-formal scheme locally of finite type whose special fiber $\text{Gr}^{PR}:=\widetilde{\mathcal{F}}\otimes \O_E/\pi$ is Pappas-Rapoport twisted affine Grassmanian associated to an hyperspecial subgroup (\cite{Rapoport_Twisted}).
\item There is an isomorphism 
$$
\widetilde{\mathcal{F}} \simeq \Gr^{PR}\hat{\otimes}_{\Fq} \O_E.
$$
\item Let us note $\widetilde{\mathcal{F}}_\eta$ the associated ind-rigid analytic space that is the generic fiber of $\widetilde{\mathcal{F}}$. Then 
$$
\text{Gr} = \underset{\text{Frob}}{\limp} \widetilde{\mathcal{F}}_{\eta}
$$
is an ind perfectoid space that is the perfection of $\widetilde{\mathcal{F}}_{\eta}$.
\end{enumerate}
\end{prop}

Point (1) is deduced from point (2). Point (2) is a consequence of the fact that for $R\in \text{Nilp}_{\O_E}$ with $\pi\mapsto a\in R$ the correspondence $T\mapsto T-a$ defines an automorphism of  $R[T]$
that extends to a continuously to an automorphism of $R\llbracket T\rrbracket$ since $a$ is nilpotent. 

If $Z$ is a proper scheme over $\Fq$ then $(Z\hat{\otimes}_{\Fq} \O_E)_\eta =(Z\otimes_{\Fq} E)^{rig}$. 
Since Pappas-Rapoport twisted affine Grassmanian is an inductive limit of proper schemes we deduces the following.

\begin{coro} Suppose $G$ is unramified. 
The $B_{dR}$-affine Grassmanian $\text{Gr}$ is the ind-perfectoid space that is the perfection of the analytification of the scalar extension from $\Fq$ to $E$ of the Pappas-Rapoport twisted affine Grassmanian associated to a hyperspecial subgroup:
$$
\Gr^{B_{dR}}= \underset{Frob}{\limp} (\Gr^{PR}\otimes_{\Fq} E)^{rig}.
$$
\end{coro}

One can deal easily with the general case using descent theory from the unramified case. In this proposition there is no hypothesis on $G$.

\begin{prop}
For $R$ an affinoid $E$-algebra equip $R\llbracket T\rrbracket$ with the $E$-algebra structure given by $E\ni \pi \mapsto T+\pi$. 
\begin{enumerate}
\item The étale sheaf associated to the presheaf $R\mapsto G ( R \llparent T\rrparent )/G(R\llbracket T\rrbracket)$ is representable by an ind-$E$-rigid analytic space.
\item The perfection of this ind-rigid-analytic space is an ind-perfectoid space that represents the $B_{dR}$-affine grassmanian.  
\end{enumerate}
\end{prop}

%
%

\begin{rema}
On has to be careful that 
$$
\pi_0 ( \Gr^{B_{dR}}) = \pi_1 (G)/\Gamma.
$$
Of course, when $G$ is unramified this coincides with \cite{Rapoport_Twisted}. But for ramified $G$ this is different from \cite{Rapoport_Twisted}.
\end{rema}

\subsubsection{The equal characteristic case}

Suppose now that $E|\Qp$.
If $C|E$ is complete algebraically closed there is a bijection
$$
X_*(T)^+ \iso G(B^+_{dR}(C))\bc G(B_{dR}(C))/G(B^+_{dR}(C)).
$$
This defines a relative position application 
$$
|\text{Gr}|\ldrt X_*(T)^+/\Gamma
$$
that is semi-continuous. We can then write 
$$
\text{Gr} = \underset{\mu\in X_*(T)^+/\Gamma}{\limi} \text{Gr}^{\leq \mu}.
$$
We then have the following theorem of Scholze.

\begin{theo}[\cite{ScholzeBerkeley}]
For each $\mu\in X_*(T)^+/\Gamma$ the pro-étale sheaf $\text{Gr}^{\leq \mu}$ is a diamond over $\spa (E)^\diamond$.
\end{theo}

The proof given by Scholze uses an non constructive approach via the so called faithfull topology. Nevertheless, {\it one can hope that using Quot-diamonds techniques as in section \ref{sec:defi and prop of Bung} we can find another more constructive proof of this theorem} (via the link between the $B_{dR}$-affine Grassmanian and the Hecke stack). 
\\

Here is another hope inspired by the equal characteristic case.

\begin{Hope}
If $G$ is unramified there is an object sitting over $\spa (\O_E)^\diamond$ whose special fiber over $\Fq$ is the Witt vectors affine Grassmanian as defined by Zhu 
(\cite{ZhuAffineGrassmanian}, \cite{ScholzeBhattGrass})
and whose generic fiber over $E^{\diamond}$ is the $B_{dR}$-affine Grassmanian.  
\end{Hope}

\begin{rema}
One has to be careful since {\it $\spa (\O_E)^\diamond$ has no geometric structure à priori}, this is just a pro-étale sheaf but not a diamond.
\end{rema}

\subsection{The Hecke stack}

Fix a $\mu\in X_*(T)^+/\Gamma$ and define the following.

\begin{defi}
For $(R,R^+)$ a perfectoif affinoid $\Fq$-algebra $$Hecke^{\leq \mu} (R,R^+)$$ is the groupoid of  quadruples $(\E_1,\E_2,R^\sharp, f)$ where:
\begin{itemize}
\item $R^\sharp$ is an untilt of $R$ over $E$
\item $\E_1$ and $\E_2$ are $G$-bundles on $X_R^{sch}$
\item if $D$ is the Cartier divisor on $X_R^{sch}$ defined by the untilt $R^\sharp$ then  
$$
f:\E_{1|X_R^{sch}\setminus D} \iso \E_{2|X_R^{sch}\setminus D}.
$$
\item Fiberwise over $\spa (R,R^+)$, the modification $f$ is bounded by $\mu$.
\end{itemize}
\end{defi}

Let us make the last condition more precise. For $\spa (F)\drt \spa (R,R^+)$ a geometric point, the untilt $R^\sharp$ defines an untilt $C$ of $F$ that gives a closed point $\infty \in X^{sch}_F$ with $\widehat{\O}_{X^{sch}_F,\infty}=B^+_{dR}(C)$.
For $i=1,2$, let us fix a trivialization of 
$\E_i\underset{X^{sch}_F}{\times} \spec (B^+_{dR}(C))$. Then the pullback of $f$ to $\spec (B_{dR}(C))$ is given by an element of $G(B_{dR}(C))$. Up to the choice of the preceding trivialization this defines en element of 
$$
X_*(T)^+= G(B^+_{dR}(C))\bc G(B_{dR}(C))/G(B^+_{dR}(C)).
$$
We ask that this element is bounded by en element in the Galois orbit $\mu$.

\begin{rema}
We could have defined directly $Hecke^{\leq \mu} (S)$ for any $S\in \text{Perf}_{\Fq}$ but this definition is more subtle. In fact one has to use the adic curve $X_S^{ad}$ for which the definition of a modification between two $G$-bundles is more complicated since one has to impose that the isomorphism on $X_S^{ad}\setminus D$ is "meromorphic along $D$". Here we thus use a great advantage of the schematical curve: the definition of a modification is more simple (and Beauville-Laszlo gluing applies directly to it contrary to the adic curve where it can be applied but this is a little bit more complicated).
\end{rema}

\begin{defi}
We define the following diagram
$$
\xymatrix{
 & Hecke^{\leq \mu} \ar[ld]_{\overset{\leftarrow}{h}} \ar[rd]^{\hd}
  \\
 \Bun_G & & \Bun_G \times \spa (E)^\diamond.
}
$$
by the formulas $\hg ( \E_1,\E_2,R^\sharp,f)= \E_2$ and $\hd (\E_1,\E_2,R^\sharp,f)= (\E_1,R^\sharp)$.
\end{defi}

Define $L^+G$ as a group on $E^{\diamond}$ by the formula 
$$
L^+G(R,R^+) = G ( B^+_{dR}(R))
$$
for $(R,R^+)$ a perfectoid affinoid $E$-algebra. There is an étale $L^+G$-torsor 
$$
\mathcal{T}\ldrt \text{Bun}_G\times \spa (E)^\diamond.
$$
In fact, given  a perfectoid $\Fq$-algebra $R$ with an untilt $R^\sharp$ over $E$, 
for $\E$ a $G$-bundle over $X_R^{sch}$ its pullback to $\spec (B^+_{dR}(R^\sharp))$
defines a $G$-torsor over $\spec (B^+_{dR}(R^\sharp))$.

The group $L^+G$ acts on the $B_{dR}$-affine grassmanian and we have the following proposition that uses again Beauville-Laszlo gluing.

\begin{prop}\label{prop:loc fibration Gr}
The morphism $\hd:Hecke^{\leq \mu}\drt \Bun_G\times \spa (E)^\diamond$ is identified with 
$$
\mathcal{T}\underset{L^+ G}{\times} \text{Gr}^{\leq \mu}.
$$
\end{prop}

Said in another way, {\it  $\hd$ is an étale locally trivial fibration in $\text{Gr}^{\leq \mu}$ with  gluing morphisms given by elements in $L^+G$.}

\subsection{A hypothetical geometric Satake correspondence}

Let us fix $\ell\neq p$. Let us note $\,^L G= \widehat{G}\rtimes \Gamma$ for the $\Qlb$-Langlands dual of $G$.

\begin{conj}\label{conj:satake geo}
\begin{enumerate}
\item 
There is a category of $L^+G$-equivariant $\Qlb$-"perverse sheaves" on $\text{Gr}$ that is equivalent to $\text{Rep}_{\Qlb} (\,^L G)$, representations of $\,^L G$ in finite dimensional $\Qlb$-vector spaces whose restriction to $\widehat{G}$ are algebraic and that are discrete on the $\Gamma$-factor of $\,^L G$. 
\item 
Via this equivalence, for  $\mu \in X_*(T)^+/\Gamma$, if $\mu'\in \mu$ with $\text{Stab}_{\Gamma}(\mu')=\Gamma'$, the "intersection cohomology complex" $IC_\mu$ of $\text{Gr}^{\leq \mu}$ corresponds to 
$r_\mu:=\text{Ind}^{\,^ L G}_{\widehat{G}\rtimes \Gamma'} r_{\mu'}$ where $r_{\mu'} \in \text{Rep}_{\Qlb} ( \widehat{G}\rtimes \Gamma')$ is the highest weight $\mu'$ irreducible representation.
\item If $\mu$ is minuscule then $IC_\mu= \Qlb (\langle \rho,\mu\rangle )[\langle 2\rho,\mu\rangle]$ where $\rho$ is the half sum of the positive roots of $T$.
\end{enumerate}
\end{conj}

Thanks to proposition \ref{prop:loc fibration Gr} the intersection cohomology complex $IC_\mu$ should define a "perverse sheaf" still denoted $IC_\mu$ on $\text{Hecke}^{\leq \mu}$ (for this one will need the perversity of $IC_\mu$ so that this satifies descent with respect to étale descent data).
\\

At the end the couple 
$$
(Hecke^{\leq \mu},IC_{\mu})
$$
allows us to define 
{\it cohomological correspondences between sheaves on $\Bun_G$ and on $\Bun_G\times \spa (E)^\diamond$}.

\section{Statement of the conjecture}
\label{sec:Statement of the conjecture}

As before $G$ is quasi-split over $E$ and $\Gamma=\Gal (\Eb|E)$ and we fix $\ell\neq p$.
 We note $\,^L G= \widehat{G}\rtimes \Gamma$ for its $\Qlb$-Langlands dual. Let us note  that this should be defined intrinsically via the conjecture \ref{conj:satake geo}. A Langlands parameter is by definition a continuous morphism
 $$
 \ph:W_E\ldrt \,^L G
 $$
 whose projection toward the $\Gamma$-factor is given by the embedding $W_E\subset \Gamma$. For such a parameter we note
 $$
 S_\ph = \big \{ g\in \widehat{G}\ |\  g\ph g^{-1} =\ph\big \}.
 $$
 This is the $\Qlb$-points of an algebraic group whose neutral connected component is reductive. We always have 
 $$
 Z(\widehat{G})^{\Gamma}\subset S_\ph.
 $$
\begin{defi}
\begin{enumerate}
\item We say $\ph$ is discrete if $S_\ph/Z(\widehat{G})^\Gamma$ is finite.
\item We say $\ph$ is cuspidal if $\ph$ is discrete and moreover the image of the associated $1$-cocyle $I_E\drt \widehat{G}$ is finite.
\end{enumerate}
\end{defi}

Via an hypothetical Langlands correspondence:
\begin{itemize}
\item discrete parameters should parametrize discrete series L-packets
\item cuspidal parameters should parametrize supercuspidal L-packets, that is to say packets all of whose components are supercuspidal.
\end{itemize} 

In all known cases of the local Langlands correspondence those two properties are satisfied.

\begin{rema}
In this context there is no notion of tempered L-parameters. In fact such a notion would depend on the choice of an isomorphism between $\Qlb$ and $\C$. One can define such an arithmetic  notion if we impose that the eigenvalues of the image of Frobenius under $\ph$ are Weil numbers. But this last notion can not be interpolated $\ell$-adically and is not natural from a purely local point of view.
\end{rema}

There is a groupoid of Langlands parameters 
$$
\mathcal{L}_G= \Big [ \Hom ( W_E,\, ^L G) /\widehat{G}\Big ].
$$
From this point of view discrete parameters corresponds  to the points of this "stack of parameters where it is Deligne-Mumford, up to the $Z(\widehat{G})^\Gamma$-factor."

\begin{conj}[rough version]
There is a functor between groupoids 
\begin{eqnarray*}
\mathcal{L}_G^{disc} & \ldrt & \text{Perverse Weil-sheaves on }\Bun_G \\
\ph & \longmapsto & \F_\ph
\end{eqnarray*}
satisfying the following properties:
\begin{enumerate}
\item The stalks of this functor at the residual gerbes at semi-stable points of $\text{Bun}_G$ defines a local Langlands correspondence for extended pure inner forms of $G$.
\item If $\ph$ is cuspidal then the restriction of $\F_\ph$ to the non-semi-stable locus of $\Bun_G$ is zero.
\item $\F_\ph$ is an Hecke eigenvector with eigenvalue $\ph$ 
\end{enumerate}
\end{conj}

We are now going to give a more precise formulation. Before beginning, we have to {\it fix a Whittaker datum}. In fact the construction of $\F_\ph$ has to depend on such a choice since the local L-packet of $G$ constructed via $\F_\ph$ will have a distinguished element given by the trivial representation of $S_\ph$. This element has to be the unique generic element of the L-packet.

\begin{conj}
Given a discrete Langlands parameter $\ph$ there is a "$\Qlb$-perverse Weil sheaf" $\F_\ph$ on $\Bun_{G,\Fqb}$ equipped with an action of $S_\ph$ satisfying the following 
properties:
\begin{enumerate}
\item For $\a\in \pi_1(G)_\Gamma$ the action  of $Z(\widehat{G})^\Gamma$ on 
the restriction of $\F_\ph$ to the component 
$\Bun_{G,\Fqb}^\a$ is given by $\a$ via the identification $\pi_1(G)_\Gamma=X^* \big  (Z(\widehat{G})^\Gamma \big )$.
\item {\it (Cuspidality condition)} If $\ph$ is moreover cuspidal then $\F_\ph = j_! j^*\F_\ph$ where $j:\Bun_{G,\Fqb}^{ss}\hookrightarrow \Bun_{G,\Fqb}$.
\item {\it (Realization of local Langlands)} For $b\in G(L)$ basic and $x_b:[\spa (\Fqb)/J_b(E)]\hookrightarrow \Bun_{G,\Fqb}$, as a representation of $S_\ph\times J_b(E)$, smooth on the $J_b(E)$-component, 
$$
x_b^* \F_\ph = \bigoplus_{\rho\in \widehat{S_\ph}\atop \rho_{|Z(\widehat{G})^\Gamma} = \kappa (b)} \rho \otimes \pi_{\ph,b,\rho}
$$
where $\big \{ \pi_{\ph,b,\rho}\big \}_{\rho}$ is an L-packet defining a local Langlands correspondence for the extended pure inner form $J_b$ of $G$. 
\\
For $b=1$, $\pi_{\ph,1,1}$ is the unique generic element of this L-packet associated to the choice of the Whittaker datum.
\item {\it (Hecke eigensheaf property)} For $\mu\in X_*(T)^+/\Gamma$, there is an isomorphism 
$$
\hd_! \Big ( \hg^* \F_\ph \otimes IC_{\mu} \Big ) \simeq \F_\ph \boxtimes r_{\mu} \circ \ph.
$$
Here $r_{\mu}\circ\ph$ is an $\ell$-adic representation of $W_E$ that defines a Weil-étale local system on $\spa (E)^\diamond\times \spa (\Fqb)=\spa (\widehat{E^{un}})^{\diamond}$. 
This isomorphism is compatible with the action of $S_\ph$ where the action of $S_\ph$ on $r_\mu\circ \ph$ is the induced by the one on $\ph$.
\item {\it (Character sheaf property)} For $\delta \in G(E)$ elliptic seen as an element of $G(L)$ the action of Frobenius on $x_\delta^* \F_\ph$ coindices with the action of $\delta \in J_\delta (E)$.
\item {\it (Local global compatibility)} 
Let $(H,X)$ be a Hodge type Shimura datum with $H_{\Qp}=G$ and $\Pi$
an automorphic representation of $H$ such that $\ph_{\Pi_p}=\ph$.   
There is a compatibility between 
the $\Pi^p$-isotypic component of 
Caraiani-Scholze sheaf $R\pi_{HT*}\Qlb[\dim \Sh]$ (\cite{CaraianiScholze}) and a multiple of the restriction of $\F_\ph$ to
the Hodge-Tate period Grassmanian. 
\end{enumerate} 
\end{conj}

We will make point (5) et (6) more precise later.
Let us note that via Beauville-Laszlo morphism
$$
\mathcal{BL}:\Gr\ldrt \Bun_G
$$
the "perverse sheaf" $\F_\ph$ should correspond to a "perverse sheaf" on 
$\Gr$. But contrary to $IC_\mu$ and the other perverse sheaves showing up 
in the geometric Satake isomorphism, this complex of sheaves won't be locally constant 
on open Schubert cells and will be of a much more complicated nature.

\begin{rema}
If the center of $G$ is connected all inner forms of $G$ are extended pure inner forms, that is to say of the form $J_b$ with $b$ basic. That being said this is false in general, for example for $\SL_n$ where the only extended pure inner form is $\SL_n$ itself. For $\SL_n$ the obstruction to reach all inner forms lies in the Galois cohomology group $H^2 (E,\mu_n)$ and is given by the fundamental class of local class field. Motivated by section prop.3.4 of \cite{Gtorseurs} that says that this fundamental class corresponds to the first Chern class of $\O(1)$ one can extend the preceding conjecture to all inner forms for $\SL_n$ by looking at the stack of $\SL_n$-bundles on the gerb of $n$-th roots of $\O(1)$. The author hopes a similar approach will lead to a generalization of the preceding conjecture to all inner forms for all reductive groups $G$ whose center is not connected. 
\end{rema}

\begin{rema}\label{rem:fibre aux points non basiques}
If one takes the fiber $x_b^*\F_\ph$ at some non-basic $b$, this is a representation of 
$\widetilde{J}_b$ that has to factorize through $\pi_0 (\widetilde{J}_b)=\underline{J_b (E)}$ (incompatibility between $\ell$ and $p$). It seems logical to think that this smooth representation of $J_b (E)$ should be linked to some Jacquet functor associated to the Levi subgroup $M_b$, the centralizer of the slope morphism (see rem.\ref{rem:localGlobalCusp} for further comments). 
\end{rema}

\section{The character sheaf property}
\label{sec:The character sheaf property}

Recall that $\F_\ph$ is a Weil sheaf and is thus equipped with a Frobenius descent datum from $\Bun_{G,\Fqb}$ to $\Bun_G$. Let us pick $\delta \in G(E)$ and let us see it as an element of $G(L)$. The morphism $\nu_\delta:\DD\drt G$ is defined over $E$ since the Dieudonné-Manin slope decomposition is defined over $E$ (given an automorphism of a finite dimensional $E$-vector space split its characteristic polynomial as a product according to the absolute values of its roots). Moreover $\delta$ lies in the centralizer of $\nu_\delta$. As a consequence, {\it if $\delta$ is elliptic in $G(E)$ it is basic in $G(L)$.} We thus have a functor between groupoids
$$
\big [ G(E)_{ell}/\text{conjugacy}\big ] \ldrt \big [ G(L)_{basic}/\s\text{-conjugacy}\big ]
$$
and thus an application
$$
\big \{ G(E)\big \}_{ell}\ldrt B(G)_{basic}.
$$

We now suppose {\it $\delta$ is elliptic}. The morphism
$$
x_\delta : \spa (\Fqb)\ldrt \Bun_{G,\Fqb}
$$
is then defined over $\Fq$. Using the Weil-sheaf structure of $\F_\ph$, $x_\delta^* \F_\ph$ is equipped with an action of a Frobenius $\text{Frob}$.

Via the isomorphism 
$$
\big [\spa (\Fqb)/J_\delta (E) \big ] \iso \Bun_{G,\Fqb}^{ss,\kappa (\delta )}
$$
given by $x_\delta$ 
the Frobenius descent datum on the right is given by the couple of morphisms 
$$
(Frob, \s): (\spa (\Fqb), J_\delta (E))\ldrt (\spa (\Fqb),J_\delta (E))
$$
on the left, 
where $\s$ is seen as an automorphism of $G(L)$ that restricts to an automorphism of $J_\delta (E)$. But now, for $g\in J_{\delta} (E)$, $g\delta g^{-\s}=\delta$ with $g\in G(L)$, one has 
$$
g^\s = \delta^{-1} g \delta.
$$
One deduces that the Frobenius descent datum is given by 
$$
(Frob,Int_{\delta^{-1}}): (\spa (\Fqb),J_\delta (E))\ldrt (\spa (\Fqb),J_\delta (E)).
$$
Thus, the action of $Frob$ on the $\Qlb$-vector space $V=x_\delta^* \F_\ph$ is given by an 
automorphism $u\in \GL (V)$ satisfying 
$$
\forall g\in J_\delta (E),\ u\circ \pi (g) = \rho (\delta)^{-1} \circ \pi (g)\circ \rho (\delta)\circ u
$$
where $\pi:J_\delta (E)\drt \GL (V)$. {\it The character sheaf property then asks that 
$u=\pi (\delta)$}.
\\

Here is a rephrasing of this condition. 
There is a decomposition 
$$
\pi=\bigoplus_{\rho\in \widehat{S_\ph}\atop \rho_{|Z(\widehat{G})^\Gamma=\kappa (\delta)}} \rho\otimes \pi_\rho
$$
that commutes with the action of $Frob$ (by hypothesis, the Weil descent datum commutes
with the action of $S_\ph$). Since by hypothesis for each $\rho$ the representation $\pi_\rho$ is irreducible, there exists a collection of scalars $(\l_\rho)_{\rho}$ in $\Qlb^\times$ such that the action of $Frob$ is given by $\oplus_\rho \l_\rho \pi(\delta)$. {\it The character sheaf property then asks that for all $\rho$, $\l_\rho = 1$}.
\\

Here is a consequence of the introduction of the character sheaf property.
\\

{\bf Consequence of the character sheaf property:} 
{\it 
Let $T_\ph$ be the stable distribution on $G(E)$ associated to the Langlands parameter $\ph$. The restriction of $T_\ph$ to the elliptic regular subset in $G(E)$ is given by the trace of Frobenius function 
$$
\delta \longmapsto e(J_\delta) \text{Tr} ( Frob ; x_\delta^* \F_\ph)
$$
via the map $\{G(E)\}_{ell}\drt B(G)_{basic}$ where $e(J_\delta)\in \{\pm 1\}$ is Kottitz sign (\cite{KottwitzSign}).
}
\\

Here one has to be careful with the preceding formula. In fact, when $J_\delta$ is anisotropic modulo its center $x_\delta^*\F_\ph$ is finite dimensional and the trace makes sense. But when this is not the case one has to interpret it as $e(J_\delta)$ times the character of the finite lenght representation $x_\delta^*\F_\ph$ evaluated at $\delta$. More precisely this is defined as the trace of $\frac{1}{\text{vol}(K\delta K)} \boldsymbol{1}_{K\delta K}$ for $K$ sufficiently small.
Of course this may seem artificial in this second case but nevertheless, the fact that this stable character can be interpreted as a trace of Frobenius function fascinates the author.

\begin{rema}
The author has tried to incorporate Kottwitz sign $e(J_\delta)$ in a cohomological shift or elsewhere but was unable to fix something that makes sense. At some point the choice of a sign linked to Frobenius already appeared before when we said that $IC_\mu = \Qlb (\langle \rho,\mu\rangle) [\langle 2\rho,\mu\rangle]$ for $\mu$ minuscule, a formula that involves the choice of a square root $\Qlb ( \tfrac{1}{2})$ of $\Qlb (1)$. All of this seems to be a delicate question.
\end{rema}

\begin{rema}
Of course, if one takes the trace of Frobenius times any element of $S_\ph$ (instead of $1\in S_\ph$) one obtains endoscopic distributions on $G$.
\end{rema}

\begin{rema}
Let $n\geq 1$ and $E_n|E$ the degree $n$ unramified extension of $E$.  There is a map $G(E_n)/\s\text{-conj.}\drt B(G)$. Consider an element $\delta \in G(E_n)$ whose stable conjugacy class in $G(E)$ given by its norm $\delta \delta^\s\dots \delta^{\s^{n-1}}$ is elliptic regular. Then $[\delta ]\in B(G)$ is basic and $x_\delta : \spa (\mathbb{F}_{q^n})\drt \Bun_G$. One can strengthen the preceding character sheaf property by asking that $$\text{Tr} ( Frob_{q^n}; x_{\delta}^*\F_\ph)
=(\text{BC}_{E_n/E} T_\ph) (\delta),$$
 the value at $\delta$ of the base change of the stable character associated to $\ph$.
\end{rema}

\begin{rema}
The character sheaf property is inspired by the method employed in \cite{ScholzeLanglandsLocal}, taking the trace of Frobenius on deformation spaces of $p$-divisible groups to define functions on $p$-adic groups.
\end{rema}

\section{F-Isocrystals, $p$-divisible groups and modifications of vector bundles}
\label{sec:isocris BT et vect} 

In this section we give complements that we will use for sections \ref{sec:localglobal} and \ref{sec:Kottwitz conj}.

\subsection{Compactification of $Y$}
\subsubsection{The affinoid case}

Suppose $S=\spa (R,R^+)$ is affinoid perfectoid. Recall the space $Y=Y_S$ from section 
\ref{sec:the curve}. Suppose first $E=\Fq\llparent \pi\rrparent$. Then one has $Y=\DD^*_S$. This extends naturally to the an $\O_E$-adic space 
$$
\DD_S \ldrt \DD_{\Fq} = \spa (\O_E)
$$
that is nothing else than the open disk over $S$ with 
$$
\DD^*_S= \DD_S \setminus \{\pi=0\}.
$$
But in fact, using the integral structure given by $R^+$ on $R$ this extends to a bigger adic space 
\begin{eqnarray*}
\mathcal{Y} &=& \spa ( R^\circ \llbracket \pi\rrbracket, R^+ + R^\circ \llbracket \pi\rrbracket )_a \\
&=& \spa ( R^\circ \llbracket \pi\rrbracket, R^+ + R^\circ \llbracket \pi\rrbracket ) 
\setminus \{ \pi=0 \} \cup \spa ( R^\circ \llbracket \pi\rrbracket, R^+ + R^\circ \llbracket \pi\rrbracket )  \setminus \{ \varpi_R=0 \}
\end{eqnarray*}
where the subscript "a" means we take the analytic points and $\varpi_R\in R^{\circ\circ}\cap R^\times$. Recall that in this definition, $R^\circ \llbracket \pi\rrbracket$ is equipped with the $(\varpi_R,\pi)$-adic topology. The function $\delta$ given by formula (\ref{eq:fonction delta equal char}) of section \ref{sec:courbe unequal} extends to a function 
$$
\delta : |\mathcal{Y}|\ldrt [0,1].
$$
{\it The space $\mathcal{Y}$ is some kind of compactification of $Y$ over $\O_E$ obtained by adding the divisors $(\pi)$ and $(\varpi_R)$} and 
$$
Y= \delta^{-1} ( ]0,1[).
$$
Note $E^{1/p^\infty}$ the $\pi$-adic completion of the perfection of $E$.
One has 
$$
\mathcal{Y}\hat{\otimes}_{\O_E} \O_E^{1/p^\infty} = \spa ( R^\circ \llbracket \pi^{1/p^\infty}\rrbracket, R^+ + R^\circ \llbracket \pi^{1/p^\infty} \rrbracket )_a.
$$
This is a perfectoid space over $\O_{E}^{1/p^\infty}$.
The open subset $\{\pi\neq 0\}$ is a perfectoid space over the perfectoid field $E^{1/p^\infty}$ and $\{\varpi_R\neq 0\}$ is a perfectoid space over the perfectoid field $\Fq \llparent T^{1/p^\infty}\rrparent$ via $T\mapsto \varpi_R$.
 In this sense $\mathcal{Y}$ is preperfectoid. 
\\

Suppose now that $E|\Qp$. The fact is that the same formula 
\begin{eqnarray*}
\mathcal{Y} &=& \spa (\A, [R^+] + \pi \A)_a \\
&=& \spa (\A, [R^+] + \pi \A) \setminus V(\pi) \cup \spa (\A, [R^+] + \pi \A) \setminus V([\varpi_R])
\end{eqnarray*}
still defines an $\O_E$-adic space. Here $\A = W_{\O_E} (R^\circ)$ and 
$$
\A^+ =\big  \{\sum_{n\geq 0} [x_n]\pi^n \in \A\ |\ x_0\in R^+\big \}.
$$
As before, this is a compactification of $Y$ by the divisors $(\pi)$ and $([\varpi_R])$. Moreover, if $E_\infty$ is the completion of the extension generated by the torsion points of a Lubin-Tate group then $\mathcal{Y}\hat{\otimes}_{\O_E} \O_{E_\infty}$ is perfectoid with tilting the preceding equal-characteristic space associated to $\O_{E_\infty}^\flat$.

\subsubsection{The general case}

Let's come back to the equal characteristic case.
For an interval $I\subset [0,1]$ (different from $\{0\}$ and $\{1\}$) we note $\mathcal{Y}_I$ for the corresponding annulus in $\mathcal{Y}$ defined via the radius function $\delta$.  
 The formula 
$$
\DD_S = \mathcal{Y}_{S,[0,1[} \subset  \mathcal{Y}_{S,]0,1[}= \DD^*_S
$$
tells us that the preceding space $\mathcal{Y}_{S,[0,1[}$ globalizes for any $\Fq$-perfectoid space $S$, not necessarily affinoid perfectoid. {\it Nevertheless this is not the case for $\mathcal{Y}_S$ in general}. For this {\it one needs an integral model of $S$. } 

More precisely, consider a perfect $\Fq$-formal scheme $\mathscr{S}$ that has an open covering by affine subsets $\spf (A)$ with $A$ $\varpi_A$-adic for some regular $\varpi_A\in A$. One can then define its generic fiber 
$\mathscr{S}_\eta$ as an $\Fq$-perfectoid space. For $A$ as before 
$\spf (A)_\eta = \spa (A[\tfrac{1}{\varpi_A}], A^+)$ where $A^+$ is the integral closure of $A$ inside $A[\tfrac{1}{\varpi_A}]$ (this is almost equal to $A$). 

Now given such an $A$ on a can define an $\O_E$-adic space 
$$
\mathcal{Y}_{\mathscr{S}}
$$
such that 
$$
Y_{\mathscr{S}_\eta} = \mathcal{Y}_{\mathscr{S},]0,1[}.
$$
This works for any $E$ of equal or unequal characteristic. We won't enter too much into the details and restrict to the affinoid case in the following to simplify. 

\subsection{F-isocrystals and vector bundles}
\label{sec:F isoc and vb}

Suppose now $E=\Qp$ and let $S=\spa (R,R^+)$ be affinoid perfectoid. Fix some pseudo-uniformizing element $\varpi_R$. {\it The link between crsytalline $p$-adic Hodge theory and the curve} is given by the following. Note $\mathcal{Y}:=\mathcal{Y}_{S}$. Consider the ring 
$$
B^+_{cris} ( R^\circ/\varpi_R) = H^0_{cris} ( \spec ( R^\circ/\varpi_R)/\spec (\Zp), \O ) \unp
$$
where $\O$ is the structural sheaf on the crystalline site. Since the Frobenius on $R^\circ/\varpi_R$ is surjective this site has a final object given by Fontaine's ring $A_{cris} (R^\circ/\varpi_R)$. This is the $p$-adic completion of 
$$
W ( R^\circ) \Big [ \frac{[\varpi_R^n]}{n!}\Big ]_{n\geq 1}.
$$
One then has 
$$
B^+_{cris} = A_{cris} \unp.
$$
The fact is now that a simple computation gives the following:
$$
\GG ( \mathcal{Y}_{[1/p^p,1]},\O) \subset 
B^+_{cris} (R^\circ/\varpi_R)\subset \GG ( \mathcal{Y}_{[1/p^{p-1},1]}, \O).
$$
Recall that the action of Frobenius on $\mathcal{Y}$ satisfies 
$$
\delta (\ph (y))=\delta (y)^{1/p}.
$$
Via the global section functor on the cristalline site 
{\it an $F$-isocrystal on $\spec (R^\circ/\varpi_R)$ is the same as a projective $B^+_{cris} (R^\circ/\varpi_R)$-module of finite type equipped with a semi-linear endomorphism whose linearization is an automorphism.} One deduces the following.

\begin{prop}\label{prop:vecto bundle ass crys}
The category of $F$-isocrystals on $(\spec (R^\circ/\varpi_R)/\spec (\Zp))_{cris}$ is equivalent to the category of $\ph$-equivariant vector bundles on $\mathcal{Y}_{(R,R^+),]0,1]}$. 
\end{prop}

\begin{rema}\label{rem:inf thinck change pas}
At the end we see that this category of $F$-isocrystals is independent of the choice of $\varpi_R$. But in fact we could already see this before since the category of $F$-isocrystals is invariant under infinitesimal thickenings.  
\end{rema}

We deduce from this a functor.

\begin{coro}
The restriction from $\mathcal{Y}_{]0,1]}$ to $Y=\mathscr{Y}_{]0,1]}\setminus V([\varpi_R])$ induces a functor 
$$
F\text{-isocrystals on } \spec (R^\circ/\varpi_R) \ldrt \Bun_{X_{R,R^+}}.
$$
\end{coro}

This functor is always fully faithfull but not essentially surjective in general. Nevertheless we have the following (\cite{Courbe}) that we won't use later (although this is used by Scholze and Weinstein in \cite{ScholzeWeinstein} that we will use...).

\begin{theo}\label{theo:equivalence cristaux bun}
When $S$ is a geometric point that is to say $S=\spa (F)$ with $F$ algebraically closed then this is an equivalence
$$
F\text{-isocrystals on } \spec (\O_F/\varpi_F) \iso \Bun_{X_{F}}.
$$
\end{theo}

This theorem is already false when $F$ is a perfectoid non-algebraically closed field. 

\begin{rema}
One has to be extremely careful with the equivalence of theorem \ref{theo:equivalence cristaux bun} since its inverse is not an exact functor and this is thus not an equivalence of exact categories.
\end{rema}

\begin{Hope}
The stack $\Bun$ of vector bundles on our curve is the stack associated to the prestack of $F$-isocrystals for the faithfull topology on $\Perf_{\Fp}$, that is to say the prestack $(R,R^+)\mapsto F$-isocrystals on $\spec (R^\circ/\varpi_R)$. This would say that in some sense $\Bun$ is the "generic fiber of the stack of $F$-isocrystals" in the sens of rigid analytic geometry.
\end{Hope}

\subsection{$p$-divisible groups and modifications of vector bundles}

Let $S=\spa (R,R^+)$ be affinoid perfectoid as before. Let now $H$ be a $p$-divisible group over $R^\circ/\varpi_R$. Consider its {\it covariant} Dieudonné-crystal $\DD (H)_\Q$. According to prop.\ref{prop:vecto bundle ass crys} one can associate to it a vector bundle 
$$
\E (\DD (H)_\Q)
$$
on $X:=X_{S}$. Suppose now that $S^\sharp = \spa (R^\sharp,R^{\sharp,+})$ is an untilt of $R$. There is an isomorphism
$$
R^{\sharp,\circ}/\varpi_R^\sharp \iso R^\circ/\varpi_R
$$
up to replacing $\varpi_R$ by $\varpi_R^{1/p^n}$ for $n\gg 0$
 (recall the change of $\varpi_R$ does not change the $F$-isocrystal $\DD(H)_\Q$, see rem. \ref{rem:inf thinck change pas}, this change is thus harmless). We note
 $$
i: S^\sharp\hookrightarrow X_S
 $$
 the corresponding Cartier divisor on $X_S$. 
 Let now $\widetilde{H}$ be a lift of $H$ to $R^{\sharp,\circ}$. 
The specialization 
$$
i^* \E ( \DD (H)_\Q)
$$
is then identified with the Lie algebra of the universal extension of $\widetilde{H}$ and is thus equipped with a Hodge filtration
$$
0\ldrt \omega_{\widetilde{H}^D}\unp \ldrt i^* \E ( \DD (H)_\Q) \ldrt \Lie  ( \widetilde{H})\unp\ldrt 0.
$$
Here is now the content of Fontaine's comparison theorem for $p$-divisible groups.

\begin{prop}
Let $V_p (\widetilde{H})$ be the pro-étale $\Qp$-local system 
on $S$ associated to the generic fiber of $\widetilde{H}$ and $V_p ( \widetilde{H})\otimes_{\Qp} \O_{X_S}$ be the corresponding slope $0$ semi-stable vector bundle on $X_S$. There is then an exact sequence 
$$
0\ldrt V_p (\widetilde{H})\otimes_{\Qp} \O_{X_S} \ldrt \E (\DD ( H)_\Q) (1)\ldrt i_* \omega_{\widetilde{H}^D}\unp\ldrt 0
$$
\end{prop}

Here the twisting by $\O_{X_S} (1)$ is due to the fact that we $\DD ( \widetilde{H}_\Q)$ is the {\it covariant} Dieudonné-module. 
We can now translate this in terms of the Hecke stack. 

\begin{defi}\label{def:le pre champ BT}
We define $\mathcal{BT}^{d,h}$ over $\spa (\Qp)^\diamond$ to be the prestack on $\Qp$-affinoid perfectoids such that $\mathcal{BT} (R,R^+)$ is the groupoid of height $h$ $d$-dimensional  $p$-divisible groups over $R^\circ$.
\end{defi}

Of course $\mathcal{BT}^{d,h}$ is not a stack. We will use it as an intermediate pre-stack to construct morphisms between two stacks.

\begin{rema}
We could have taken $p$-divisible groups over $R^+$ instead of $R^\circ$. As we said before (\ref{rem:surconv}) all the objects we are interested in 
are partially proper and this is thus harmless to impose $\mathcal{BT}$  is partially proper.
\end{rema}

We thus deduces the following.

\begin{coro}\label{coro:comparaison hecke}
Take $G=\GL_h$ and $\mu (z)= \text{diag} ( \underbrace{z,\dots,z}_{d\text{-times}},1,\dots,1)$.
There is a morphism $$u:\mathcal{BT}^{d,h}\ldrt Hecke^{\mu}$$ given by the Hodge filtration of the vector bundle associated to the $F$-isocrystal of a $p$-divisible group. It satisfies:
\begin{enumerate}
\item $\hg\circ u$ is given by the vector bundle associated to  the $F$-isocrystal of the reduction modulo $\varpi_R$ of the $p$-divisible group twisted by $\O(1)$
\item $\hd\circ u$ is given by the slope $0$-semi-stable vector bundle associated to the $\Qp$-pro-étale local system associated to the generic fiber of the $p$-divisible group (alias its rationnal Tate module).
\end{enumerate} 
\end{coro}

There is another "dual" way to describe this modification in terms of {\it Hodge-Tate periods}. In fact, with the preceding notations, there is $\Qp$-linear  morphism of pro-étale sheaves 
$$
\a_{\widetilde{H}} : V_p ( \widetilde{H})\ldrt \omega_{\widetilde{H}^D}\unp .
$$
We refer for example to the section 5 of \cite{FarguesCanonique} for the definition and properties of this morphism over a valuation ring. 
Recall the following (the proof is the same as the one of theorem 2, sec. 5.3.2, of \cite{FarguesCanonique} since Faltings integral comparison theorem works over any perfectoid ring, see \cite{MiaofenPeriodes}).

\begin{prop}\label{prop:HTseq}
There is an exact sequence of vector bundles on $S^\sharp$ 
$$
0\ldrt \omega_{\widetilde{H}}^\vee\unp \otimes \Qp(1)\xrig{\ \a_{\widetilde{H}^{D}}^\vee (1)\ } V_p(\widetilde{H})\otimes_{\Qp} \O_{S^\sharp} \xrig{\ \a_{\widetilde{H}}\ } \omega_{\widetilde{H}^D}\unp\ldrt 0
$$
\end{prop}

In the preceding proposition 
$$
V_p ( \widetilde{H})\otimes \O_{S^\sharp} = i^* \big ( V_p (\widetilde{H})\otimes \O_{X_S}\big )
$$
as a vector bundle over $S^\sharp$. 

\begin{prop}
The morphism $u:\mathcal{BT}^{d,h}\drt Hecke^{\mu}$ is induced 
by the modification of $V_p (\widetilde{H})\otimes_{\Qp} \O_{X_S}$ given by the Hodge-Tate filtration of $i^* \big ( V_p (\widetilde{H})\otimes_{\Qp} \O_{X_S}\big )$ associated to the Hodge-Tate exact sequence of proposition 
\ref{prop:HTseq}
\end{prop}

We thus have two equivalent points of view on the morphism $u$ if we focus on $\hg$ or $\hd$:
\begin{itemize}
\item {\it If we focus on $\hg$ this is defined via the Hodge-de-Rham periods}
\item {\it If we focus on $\hd$ this is defined via the Hodge-Tate periods.} 
\end{itemize}

\section{Local/global compatibility}
\label{sec:localglobal}

All the results we are going to speak about here extend to Shimura varieties of Hodge-type (\cite{CaraianiScholze}). We prefer to restrict ourselves to the PEL case to simplify the exposition.
\\

Let $(H,X)$ be a PEL type Shimura datum.
Suppose $H_{\Qp}=G$, the preceding quasi-split $p$-adic group with $E=\Qp$. We suppose moreover $H_{\Qp}$ is unramified. 
Fix a sufficiently small level $K^p$ outside $p$, an embedding of $\Qb$ into $\Qpb$ and note $E$ the corresponding $p$-adic completion of the reflex field (yes, there is conflict of notations...but the other $E$ is $\Qp$ now). We fix a hyperspecial compact subgroup inside $G(\Qp)$.
We note $\underline{G}$ for the corresponding reductive model of $G$ over $\Zp$.  For $K_p\subset \underline{G}(\Zp)$  a compact open subgroup we note 
$$
\Sh_{K_p}
$$
the rigid analytic space over $E$ that is the locus of good reduction for the universal abelian scheme in the analytification of the Shimura variety with level $K_p K^p$. If $K_p$ is our hyperspecial subgroup this is the generic fiber of the $p$-adic completion of Kottwitz integral model. We note $\mathcal{S}$ this integral model, $\widehat{\mathcal{S}}_\eta =\Sh_{\underline{G}(\Zp)}$.
\\

Recall the following theorem of Scholze (this is a rewriting, Scholze theorem is more powerful  but this is the only thing we need).

\begin{theo}[\cite{ScholzeTorsion}]
The pro-étale sheaf $\underset{K_p}{\limp} \Sh_{K_p}^\diamond$ is representable by an $E$-perfectoid space $\Sh_\infty$. Moreover this perfectoid space represents the sheaf associated to the presheaf on the big analytic site that sends $(R,R^+)$ to the quadruples $(A,\l,\iota,\bar{\eta}^p) \in 
\mathcal{S} (R^+)$ (\cite{Ko1}) together with a trivialization of $A[p^\infty]\otimes_{R^+} R$ compatible with its $\underline{G}$-structure.
\end{theo}

We now use the results of section \ref{sec:isocris BT et vect}. Thanks to the last assertion of the preceding theorem we only need to define all the morphisms that will follow on elements of $\mathcal{S}(R^+)$ for $(R,R^+)$ affinoid perfectoid over $E$. Any such element gives us a $p$-divisible group over $R^+$.
\\ 
 The covariant  $F$-isocrystal of this $p$-divisible group (with its $\underline{G}$-structure) together with its Hodge filtration  defines 
a morphism
$$
g: \Sh_{\infty} \ldrt \mathcal{BT}\xrig{ \ u\ } \Hecke^\mu
$$
where $\mu$ is deduced from the global Shimura datum via transfert from $\Qb$ to $\Qpb$. For this we use corollary \ref{coro:comparaison hecke} (the incorporation of PEL type level structures is immediate). Here we have changed the definition \ref{def:le pre champ BT} of the pre-stack  $\mathcal{BT}$ to incorporate the $\underline{G}$-structure. 
We then have a commutative diagram 
$$
\xymatrix{
& & & \Fl_{\mu}^{\diamond} \ar[ld]^-i \ar[rd] \\
\Sh_\infty^\flat \ar@/_1.2ex/[rd] \ar[rr]^-g 
\ar@/^3ex/[rrru]^-{\pi_{HT}^\diamond	} 
&& Hecke^\mu \ar[ld]_-{\hg} \ar[rd]^-{\hd} 
 & \Box & \spa (E)^\diamond \ar[ld]^-{(x_1,can)}
\\
& \Bun_G & & \Bun_G\times \spa (\Qp)^\diamond
}
$$
where:
\begin{itemize}
\item $\Fl_{\mu}=\text{Gr}^{\leq \mu}=\text{Gr}^{\mu}$ is the flag variety over $E$ associated to $\mu$
\item $\pi_{HT}$ is the Hodge-Tate period map (this diagram gives in fact a definition of $\pi_{HT}$)
\item the morphism $(x_1,can)$  is given by $x_1:\spa (\Fp)\drt \Bun_G$ induced by the trivial vector bundle and $can$ associated to $E|\Qp$
\item the right square is cartesian 
\item the factorisation of $g$ via $\pi_{HT}$ is given by the infinite level structure at $p$ on the Tate module of the universal abelian scheme that gives a trivialization of the corresponding $G$-bundle $V_p(\widetilde{H})$ with the notation of \ref{coro:comparaison hecke} and thus a factorization via the cartesian product given by the right square.
\end{itemize}

The local/global compatibility condition is then formulated in the following way.
\\

{\bf Local/global compatibility:} {\it 
Let $\Pi$ be an automorphic representation of $H$ with $\Pi_p$ discrete. 
There is a multiplicity $m\in \N$ such that 
$$
R\pi_{HT*} (\Qlb [\dim \Sh]) [\Pi^p] = m. i^*\hg^*\F_{\ph_{\Pi_p}}
$$
where the left hand side is Caraiani-Scholze perverse sheaf (\cite{CaraianiScholze}.
}

\begin{rema}
The composite 
$$
\Fl^{\diamond}\xrig{\ i\ } Hecke^\mu \xrig{\ \hg\ } \Bun_G
$$
defines a stratification of $\Fl_{\mu}$ via pullback of the HN-stratification of $\Bun_G$. This is the so called Newton stratification of $\Fl_{\mu}$, see \cite{CaraianiScholze}.
\end{rema}

\begin{rema}[Motivation for the cuspidality condition]\label{rem:localGlobalCusp}
Here is one motivation for the cuspidality condition coming from \cite{CaraianiScholze}. 
In \cite{CaraianiScholze} the authors prove that $R\pi_{HT*} (\Qlb [\dim \Sh]) [\Pi^p]$ is locally constant along the Newton stratification with fiber the cohomology of the corresponding Igusa variety (\cite{Har4}, \cite{MantovanAsterisque}). This is a reinterpretation of Mantovan's work in \cite{MantovanAsterisque} in the perfectoid setting where the nearby cycles $\Rpsi$ are replaced by $R\pi_{HT*}$. Using Shin's formula for the alternate sum of the cohomology of Igusa varieties (\cite{ShinTraceFormula} generalizing \cite{Har4}) one sees that the alternate sum of the cohomology of this Igusa variety as a representation of $J_b (E)$ is computed using a Jacquet functor from representations of $G(\Qp)$ toward the one of $M_b(\Qp)$ where $M_b$ is the Levi subgroup that is the centralizer of the slope morphism. {\it This alternate sum is thus zero if $b$ is not basic}. Coupled with the local/global compatibility condition this is a strong motivation for the cuspidality condition.
\end{rema}

\begin{rema}[Equal characteristic case]
Of course on should have the same type of local/compatibility in the equal characteristic case using moduli of Shtukas, the advantage being in this case that $\mu$ does not have to be minuscule which extends the possible cases for local/global compatibility. 
\end{rema}

\begin{rema}[$\ell=p$]
The  $p$-adic Banach sheaf $R\pi_{HT*} \widehat{\O}_{\Sh_\infty}$ and its torsion counterpart $R\pi_{HT*} \O_{\Sh_\infty}^+/p$ make sense. From this point of view one can ask wether our local conjecture has a $p$-adic counterpart in the framework of the so called $p$-adic Langlands program. If this is the case then for $b$ non-basic the action of $\widetilde{J}_b$ on $x_b^*\F_\ph$ won't factorize anymore through an action of $\pi_0 ( \widetilde{J}_b )=\underline{J_b(E)}$ (see rem.\ref{rem:fibre aux points non basiques}) this should lead to interesting phenomenon. For example for $\GL_{2/\Qp}$ and the modular curve for the ordinary $b$, $\widetilde{J}_b=\mathbb{B}^{\ph=p}$ (see ex.\ref{ex:auto ordi curve}) and one can hope to link this to spaces of $p$-adic distributions. 
\end{rema}

\section{Kottwitz conjecture on the cohomology of basic RZ spaces (\cite{Rapoport2})}
\label{sec:Kottwitz conj}

\subsection{The fundamental example: Lubin-Tate and Drinfeld spaces}
\label{sec:Drinfeld and Lubin Tate case}

Consider $G=\GL_{n}$ over the $p$-adic field $E$. And set $\mu (z)= \text{diag} (z,1,\dots, 1)$. Fix 
$$
b=  \left ( 
\begin{matrix}
0 & \hdots & 0 & \pi \\
1 & &  & 0 \\
 & \ddots && \vdots  \\
 & & 1 & 0
\end{matrix}
\right )^{-1}
$$
so that $\E_b  = \O (\tfrac{1}{n})$ and $J_b(E)= D^\times$ where $D$ is a division algebra over $E$ with invariant $1/n$. We note $\breve{E}$ for the completion of the maximal unramified extension of $E$.
\\

 There is a cartesian diagram
$$
\xymatrix{
\big [ \mathbb{P}^{n-1,\diamond}_{E} /\GL_n (E) \big ] \ar[d]_-i\ar[r] & [\spa (E)/\GL_n (E)] \ar[d]_{(x_1,Id)} \\
Hecke^\mu \ar[r]^-{\hd} & \Bun_{G}\times \spa (E)^\diamond.
}
$$
Here $x_1$ is given by the trivial vector bundle. 
In fact, to give one self a modification of the trivial vector bundle of the form 
\begin{equation} \label{eq:modifcalc}
0\ldrt \O^n \ldrt \E \ldrt i_*\F\ldrt 0,
\end{equation}
where $\F$ is locally free of rank $1$ over $R^\sharp$, is the same as to give oneself en element of $\P^{n-1}_E (R^\sharp)$ (as before, $R^\sharp$ is an untilt of $R$ and $i$ is the corresponding Cartier divisor $\spec (R^\sharp)\hookrightarrow X_R^{sch}$). The composite morphism 
$$
f:\big [ \mathbb{P}^{n-1,\diamond}_{E} /\GL_n (E) \big ]  \ldrt Hecke^\mu \xrig{\ \hg\ } \Bun_G
$$
gives the vector bundle $\E$ in the preceding modification (\ref{eq:modifcalc}). Geometrically fiberwise the isomorphism class of $\E$ is $\O^i \oplus \O ( \tfrac{1}{n-i})$ for some integer $i\in \{0,\dots,n-1\}$. The corresponding stratification of $\P^{n-1}$ is the pullback via $f$ of the HN strastification of $\Bun_G$. The open stratum correspond to the locus where $\E$ is isomorphic to $\O(\tfrac{1}{n})$ on each geometric fiber. This is the pullback via $f$ of the semi-stable locus and coincides with Drinfeld's space $\Omega$, more precisely this is the open subset 
$$
\big [\Omega^{\diamond}/\GL_n(E)\big ] \subset \big [ \mathbb{P}^{n-1,\diamond}_{E} /\GL_n (E) \big ].
$$
There is a corresponding morphism 
$$
f_{|\big [\Omega^{\diamond}/GL_n(E)\big ]}: \big [\Omega^{\diamond}_{\breve{E}}/\GL_n(E)\big ] \ldrt \Bun_{G,\Fqb}^{ss} \underset{\sim}{\xleftarrow{\ x_b\ }} [\spa (\Fqb)/\underline{D}^\times].
$$
The pullback via this morphism of the universal $\underline{D}^\times$-torsor on $[\spa (\Fqb)/\underline{D}^\times]$ is Rapoport-Zink version of Drinfeld's covers of $\Omega_{\breve{E}}$. At the top of this tower, the preceding modification (\ref{eq:modifcalc}) of $\O^n$ to $\O(\tfrac{1}{n})$ is the Hodge-Tate exact sequence associated to the universal $p$-divisible group on Lubin-Tate tower in infinite level. 
\\

Let $\ph:W_E\ldrt \GL_n (\Qlb)$ be  a discrete Langlands parameter (that is to say an indecomposable representation in this case). Let $\pi$ be the corresponding square integrable representation of $\GL_n (E)$ via the local Langlands correspondence and $\rho$ the associated representation of $D^\times$ via Jacquet-Langlands, $JL (\rho)=\pi$. According to the conjecture one has 
\begin{eqnarray*}
x_1^*\F_\ph &=& \pi \\ 
x_b^*\F_\ph &=& \rho.
\end{eqnarray*}
Applying proper base change to the Hecke property one finds that 
$$
R\GG_c ( \P^{n-1}_{\Cp}/GL_n (E), i^*\hg^*\F_\ph)(\tfrac{n-1}{2})[n-1] = \pi\otimes \ph. 
$$
This is means the smooth-equivariant cohomology complex of $i^*\F_\ph$
is equal to $\pi\otimes \ph (\tfrac{1-n}{2})$ concentrated in middle degree.

Suppose now moreover that $\ph$ is cuspidal that is to say the representation $\ph$ is irreducible. The cuspidality condition in the conjecture then says that if $j:\Omega\hookrightarrow \P^{n-1}$  
$$
i^*\hg^* \F_\ph = j_! j^* i^*\hg^* \F_\ph.
$$
But now, if $\M_\infty \drt \Omega_{\breve{E}}$ is Drinfeld's tower with infinite level, 
$$
j^* i^*\hg^* \F_\ph  = \M_\infty \underset{D^\times}{\times} \rho.
$$
The conjecture thus predicts that 
$$
R\GG_c ( \M_\infty\otimes \Cp, \Qlb) \otimes_{D^\times} \rho = \pi\otimes \ph (\tfrac{1-n}{2}) [1-n]
$$
which is the well known realization of the local Langlands/Jacquet-Langlands correspondence in the cohomology of Drinfeld tower.

\begin{rema}
One will remark that for $\ph$ discrete non cuspidal, the conjecture predicts the existence of a perverse sheaf on $[\P^{n-1}/\GL_n (E)]$ whose cohomology realizes local Langlands correspondence exactly in middle degree. This is different from the usual realization of the local Langlands correspondence for in the cohomology of Drinfeld tower (in general, for non-supercuspidal parameters, the corresponding representation does not show up in middle degree). The restriction of this perverse sheaf to $\Omega$ should be $\M_\infty \underset{D^\times}{\times}\rho$. From this point of view one can think of the cohomology of $[\P^{n-1}/\GL_n (E)]$ with coefficient in this perverse sheaf as being the intersection cohomology of $\Omega$ with coefficients in $\M_\infty \underset{D^\times}{\times}\rho$.
\end{rema}

\subsection{The general case}
\subsubsection{Local Shtuka moduli spaces}

We are going to see that the conjecture implies Kottwitz conjecture describing the supercuspidal part of the cohomology of basic Rapoport-Zink spaces. In the following we take $E=\Qp$ and thus $G/\Qp$.
\\

Fix $b\in G(L)$ and $\mu\in X_*(T)^+/\Gamma$. We are going to look at a disjoint union over a Galois orbit like $\mu$ of Rapoport-Zink spaces. The advantage is that we don't need to introduce any reflex field, this disjoint union is defined over $\Qp$ (this simplifies the notations). We note $\Qpbreve$ for $L$ when this is the base field of our moduli spaces.
\\

Let us look at the following diagram
$$
\xymatrix{
\spa (\Fpb) \ar[rd]^-{x_b} 
 & & \Hecke^{\leq \mu}\otimes \Fpb \ar[ld]_-{\hg} \ar[rd]^-{\hd} & & \spa (\Qpbreve)^\diamond \ar[ld]_-{(x_1,Id)} \\
 & \Bun_{G,\Fpb} & & \Bun_{G,\Fpb}\times \spa (\Qpbreve)^{\diamond}
}
$$
And define the associated local Shtuka moduli space as a fiber product of this diagram:
$$
\Sht (G,b,\mu) := \spa (\Fpb) \underset{x_b, \Bun_{G,\Fpb}, \hg}{\times} \Hecke^{\leq \mu}\otimes \Fpb  \underset{\hd, \Bun_{G,\Fpb}\times \spa (\Qpbreve)^\diamond, (x_1,Id)} \spa (\Qpbreve)^{\diamond}.
$$
This is a diamond over $\Qpbreve^{\diamond}=\Qp^\diamond\otimes_{\Fp}\Fpb$. It it equipped with an action of $J_b (\Qp)$ via the factorization 
$$
x_b:\spa (\Fpb) \ldrt [\spa (\Fpb)/J_b (\Qp)]\ldrt \Bun_{G,\Fpb}
$$
and a commuting action of $G(\Qp)$ via the factorization 
$$
x_1: \spa (\Fpb)\ldrt [\spa (\Fpb)/G(\Qp)]\ldrt \Bun_{G,\Fpb}.
$$

The diamond $\Sht (G,b,\mu)$ equipped with its action of $G(\Qp)\times J_b (\Qp)$ has a {\it Weil descent datum from $\Qpbreve^{\diamond}$ to $\Qp$}. In fact, $x_1$ is defined over $\Fp$. Moreover $x_b$ is not defined over $\Fp$ but its isomorphism class is. More precisely, $\text{Frob}^*x_b=x_{b^\s}$ and there is a diagram
\UseTwocells
$$
\xymatrix@R=1.8cm{
\spa (\Fpb) \rtwocell^{x_{b^\s}}_{x_b}{} & \Bun_{G,\Fpb}
}
$$
where the vertical arrow is given by the $\s$-conjugacy $b^\s = b^{-1}.b.b^{\s}$. 
\\

Of course, so that $\Sht (G,b,\mu)$ be non-empty one has to suppose
that $[b]\in B(G,\mu)$.
There are two period morphisms 
$$
\xymatrix{
 & \Sht (G,b,\mu) \ar[ld]_-{\pi_{dR}} \ar[rd]^-{\pi_{HT}} \\
 \Gr^{\leq -\mu}\otimes_{\Qp^\diamond} \Qpbreve^\diamond & & \Gr^{\leq \mu}
}
$$
induced by $\hg$ and $\hd$. The morphism $\pi_{dR}$, resp. $\pi_{HT}$, is $G(\Qp)$-invariant, resp. $J_b(\Qp)$-invariant, and commutes wit the action of $J_b(\Qp)$, resp. $G(\Qp)$.
The left hand side can be made defined over $\Qp$ if we suppose $b$ is basic and decent and we twist $\Gr^{\leq -\mu}\otimes \Qpbreve^\diamond$ via $b\s$ (in the minuscule case this gives rise to twisted forms of $G/P_\mu$ like Severi-Brauer varieties). This is just an exercise in checking compatibility with the preceding descent datum. 
\\

Here is a two step construction of $\Sht (G,b,\mu)$ via the so called admissible locus. Consider the diagram
$$
\xymatrix{
\Gr^{\leq -\mu}\otimes \Qpbreve^\diamond \ar[d]\ar[r] & \Hecke^{\leq \mu}\otimes \Fpb \ar[d]^-{\hg} \ar[rd]^-{\hd} \\
\spa (\Fpb) \ar[r]_-{x_b} & 
\Bun_{G,\Fpb} & \Bun_{G,\Fpb} \times \spa (\Qpbreve)^\diamond
}
$$
where the square is cartesian. This defines a morphism
$$
f:\Gr^{\leq - \mu}\otimes \Qpbreve^{\diamond}\ldrt  Hecke^{\leq \mu}\otimes \Fpb \xrig{\ \hd\ } \Bun_{G,\Fpb} \times \spa (\Qpbreve)^{\diamond} \xrig{ \ proj\ } \Bun_{G,\Fpb}.
$$
Define 
$$
(\Gr^{\leq - \mu}_{\Qpbreve^\diamond})^{ad} =f^{-1} \big ( \Bun_{G,\Fpb}^{1,ss}\big )
$$
the so called {\it admissible locus.} This is an open subset. There is thus a morphism
$$
(\Gr^{\leq -\mu}_{\Qpbreve^\diamond})^{ad} \ldrt \Bun_{G,\Fpb}^{1,ss} = \big [ \spa (\Fpb)/ \underline{G(\Qp)}\big ].
$$
This means there is a $\underline{G(\Qp)}$-pro-étale torsor over the admissible locus $(\Gr^{\leq -\mu}_{\Qpbreve^\diamond})^{ad}$. Then, {\it $\Sht(G,b,\mu)$ is the moduli space of trivializations of this $\underline{G(\Qp)}$-torsor}.
\\

We have the following properties:
\begin{itemize}
\item {\it $\pi_{dR}$ is a $\underline{G(\Qp)}$-torsor over $(\Gr^{\leq -\mu}_{\Qpbreve^\diamond})^{ad}$}
\item {\it $\pi_{HT}$ is a $\widetilde{J}_b$-torsor over the Newton stratum 
$\Gr^{\leq \mu,b}_{\Qpbreve^\diamond}\subset \Gr^{\leq -\mu}_{\Qpbreve^\diamond}$.}
\end{itemize}
\

Let us go back to the conjecture now. Suppose $b$ is basic. Let us write the Hecke property
$$
\hd_! \big ( \hg^* \F_\ph \otimes IC_\mu \big ) = \F_\ph\boxtimes r_\mu\circ \ph.
$$
Let us apply {\it proper base change} to the following cartesian diagram
$$
\xymatrix{
\big [ \Gr^{\leq \mu}\otimes \Qpbreve^\diamond /\underline{G(\Qp)}\big ]
\ar[r] \ar[d]_-{i} & \big [\spa (\Qpbreve)^\diamond /\underline{G(\Qp)}\big ] \ar[d]^{(x_1,Id)} \\
Hecke^{\leq \mu}\otimes \Fpb \ar[r]^-{\hd} & \Bun_{G,\Fpb}\times \spa (\Qpbreve)^{\diamond}
}
$$
Projection via the upper horizontal arrow is the smooth equivariant cohomology complex. We note $R\GG_c (\Gr^{\leq \mu}_{\Cp^\flat} /G(\Qp),-)$ for this complex as an element of the derived category of $\Qlb$-vector spaces equipped with a smooth action of $G(\Qp)$ and equipped with an action of $I_{\Qp}\subset W_{\Qp}$.  
One obtains
\begin{eqnarray*}
R\GG_c \Big ( \Gr^{\leq-\mu}_{\Cp^\flat} /G(\Qp), \big (i^* \hg^* \F_\ph \big )
\otimes IC_\mu 
\Big ) &=& x_1^*\F_\ph \otimes r_\mu\circ \ph.
\end{eqnarray*}
One has by hypothesis 
$$
x_1^*\F_\ph = \bigoplus_{\rho\in \widehat{\pi_0 (S_\ph)}} \rho\otimes \pi_\rho
$$
where $\{\pi_\rho\}_\rho$ is the L-packet of representations of $G(\Qp)$ associated to $\ph$ via local Langlands.
\\

Suppose now moreover that {\it $\ph$ is cuspidal and $b$ basic,} that is to say $[b]$ is the unique basic element of $B(G,\mu)$. 
Because $[b]\in B(G,\mu)$
there is then a factorization 
$$
 \xymatrix{
\Gr^{\leq \mu}_{\Qpbreve^{\diamond}} \ar@/_3mm/[rrd] \ar[r]^i & Hecke^{\leq \mu}_{\Fpb} \ar[r]^-{\hg} & \Bun_{G,\Fpb} \\
 \Gr^{\leq \mu,b}_{\Qpbreve^\diamond}  \ar@/_3mm/[rrd] \ar@{^(->}[u]^-j
 & & \Bun_{G,\Fpb}^{\kappa (b)} \ar@{^(->}[u] \\
 && \Bun_{G,\Fpb}^{\kappa (b),ss} \ar@{^(->}[u]
  & \ar[l]_-{x_b}^-{\sim} [\spa (\Fpb)/\underline{J_b(\Qp)}]  
}
$$
where the inclusions are open subsets. {\it Using the cuspidality condition} we then find that via the $\underline{J_b(\Qp)}$-torsor
$$
\pi_{HT}: \Sht (G,b,\mu)\ldrt \Gr^{\leq -\mu}_{\Qpbreve^\diamond}
$$
$$
i^*\hg^* \F_\ph = j_!\Big (\Sht(G,b,\mu)\underset{J_b(\Qp)}{\times} x_b^*\F_\ph \Big ).
$$
Using the formula
$$
x_b^*\F_\ph = \bigoplus_{\rho\in \widehat{S_\ph}\atop \rho_{Z(\widehat{G})^\Gamma}=\kappa (b)} \rho\otimes \pi_\rho 
$$
we find the following (one has to check the compatibility with the action of the Frobenius descent datum using the character sheaf property, this is left to the courageous reader...). To obtain this formula one uses that all
the preceding isomorphisms commute with the action of $S_\ph$.

\begin{prop}
The conjecture predicts the following. Choose $\mu\in X_*(T)^+/\Gamma$ and let $[b]\in B(G,\mu)$ be the unique basic element. Let $\ph: W_{\Qp}\drt \,^L G$ be a discrete cuspidal parameter. Given $\rho \in \widehat{S_\ph}$ satisfying $\rho_{|Z(\widehat{G})^\Gamma}=\kappa (b)$ note $\pi_\rho$ for the corresponding supercuspidal representation of $J_b(\Qp)$ of the $L$-packet associated to $\ph$. Consider the following decomposition 
$$
\hom (\rho , r_\mu\circ \ph) = \bigoplus_{\rho'\in (S_\ph / Z(\widehat{G})^\Gamma)^{\widehat{\ \ \ }} } \rho'\otimes \s_{\rho,\rho',\mu,\ph}
$$
as a representation of $S_\ph\times W_{\Qp}$. As a representation of $G(\Qp)\times W_{\Qp}$ we then have 
$$
H^\bullet_c ( \Sht (G,b,\mu)_{\Cp^\flat}, IC)\otimes_{J_b(\Qp)} \pi_\rho  \simeq \bigoplus_{\rho'\in (S_\ph / Z(\widehat{G})^\Gamma)^{\widehat{\ \ \ }} } \pi_{\rho'}\otimes \s_{\rho,\rho',\mu,\ph}.
$$
where $IC$ is the intersection cohomology complex of $\Sht (G,b,\mu)$. When $\mu$ is minuscule this is reduced to 
$$
H^i_c ( \Sht (G,b,\mu)_{\Cp^\flat},\Qlb)\otimes_{J_b(\Qp)} \pi_\rho =  \begin{cases}
0 \text{ if } i\neq d \\
\bigoplus_{\rho'\in (S_\ph / Z(\widehat{G})^\Gamma)^{\widehat{\ \  }} } \pi_{\rho'}\otimes \s_{\rho,\rho',\mu,\ph}(-d/2) \text{ if } i=d.
\end{cases} 
$$
with $d=\dim \Sht (G,b,\mu)$.
\end{prop}

\begin{rema}
In the preceding proposition we used that the representation $\hom (\rho , r_\mu\circ \ph)$ of $S_\ph$ is trivial on $Z(\widehat{G}^{\Gamma})$ since $[b]\in B(G,\mu)$ implies that the central chararacter of $r_\mu$ is given by $\kappa (b)= \widehat{\mu}_{|Z(\widehat{G})^\Gamma}$.
\end{rema}

\subsubsection{The link with moduli spaces of $p$-divisible groups}

Let us finish by explaining the link between $\Sht (G,b,\mu)$ and Rapoport-Zink spaces. We place ourselves in a PEL situation like in section \ref{sec:localglobal}. Let us fix a local unramified PEL type Rapoport-Zink datum over $\Qp$ (\cite{RZ}). Let $G$ over $\Qp$ be the corresponding reductive group and $\underline{G}$ its reductive integral model. There is fixed some $b\in G(L)$ corresponding to the covariant isocrystal of the  $p$-divisible group we deform. Moreover we have a Hodge chocaracter $\mu$ (as in the preceding section we take its Galois orbit and obtain a disjoint union of Rapoport-Zink spaces). Let us note $$\widehat{\M}$$ the corresponding Rapoport-Zink space as a formal scheme over $\spf ( \breve{\Z}_p)$.
\\
We note $$(\M_K)_{K\subset \underline{G} (\Zp)}$$ for the associated tower of rigid analytic spaces, $\M_{\underline{G}(\Zp)}= \widehat{\M}_\eta$. 
We then have the following theorem of Scholze-Weinstein (\cite{ScholzeWeinstein}).

\begin{theo}
The pro-étale sheaf $\underset{K}{\limp} \M_K^\diamond$ is representable by a perfectoid space $\M_\infty$. Moreover $\M_\infty$ represents the sheaf associated to the presheaf on the big analytic site that associates to $(R,R^+)$ over $\Qpbreve$ an element  $(H,\rho)\in \widehat{\M} (R^+)$ together with a trivialization of the étale $p$-divisible group with its $\underline{G}$-structure $H\otimes_{R^+} R$.
\end{theo}

In this statement the couple $(H,\rho)$ is the $p$-divisible group together with the rigidification of its reduction to $R^+/\varpi_R$, 
$$
\rho: \mathbb{H}\otimes_{\Fqb} R^+/\varpi_R\ldrt H\otimes_{R^+} R^+/\varpi_R
$$
is a quasi-isogeny compatible with all $\underline{G}$-structures.

\begin{rema}
In the preceding, since Rapoport-Zink spaces are partially proper, we can replace $R^+$ by $R^\circ$.
\end{rema}

\begin{rema}
Scholze-Weinstein theorem is stronger, stating that $\M_\infty \sim \underset{K}{\limp} \M_K$ and thus some algebras of functions in the tower are dense in some algebras of functions on $\M_\infty$.
\end{rema}

We use the results of section \ref{sec:isocris BT et vect}. There is a morphism 
$$
g:\M_\infty \ldrt \mathcal{BT}\xrig{\ u\ } Hecke^\mu
$$
given by the preceding modular description of $\M_\infty$. Here, again, we have modified the definition \ref{def:le pre champ BT} of $\mathcal{BT}$ to incorporate the $\underline{G}$-structure. It is now immediate to check that $g$ factorizes through the fiber product defining $\Sht (G,b,\mu)$ using corollary \ref{coro:comparaison hecke}:
\begin{itemize}
\item $\hg\circ g$ factorizes through $x_b:\spa (\Fpb)\drt \Bun_{G,\Fpb}$ thanks to the rigification $\rho$ that identifies the vector bundle associated to the $F$-isocrystal of the $p$-divisible group over $R^+/\varpi_R$ with $\E_b$
\item $\hd\circ g$ factorizes through $x_1$ thanks to the trivialization of the generic fiber of the $p$-divisible group.
\end{itemize}

The following is now a reinterpretation of Scholze-Weinstein results.

\begin{theo}[Scholze-Weinstein]
The morphism $\M_\infty \xrig{\ g\ } \Sht (G,b,\mu)$ is an isomorphism. 
\end{theo}

\begin{rema}
This theorem does not give a classification of $p$-divisible groups over $R^+$ for $(R,R^+)$ affinoid perfectoid over $\Qpbreve$ because of the "sheaf associated to the presheaf". Nevertheless this gives a classification over geometric points, that is to say $\O_C$ with $C|\Qpbreve$ algebraically closed. That being said, this classification is exactly what Scholze and Weinstein prove first in \cite{ScholzeWeinstein}
to prove their result: $g$ is a bijection on geometric points.
\end{rema}

\section{The abelian case}
\label{sec:The abelian case}

\subsection{The $\GL_1$-case}

In this section we check the conjecture for $\GL_1$ over a $p$-adic field $E$. In this case there is a decomposition 
$$
\Bun= \coprod_{d\in \Z} \Bun^d
$$
according to the degree $d$ of a line bundle. Let us fix a uniformizing element $\pi$ of $E$. This gives rise to a canonical line bundle $\O(1)$ on the curve which corresponds to $\pi^{-1}\in \GL_1(L)$. This choice being fixed there is an isomorphism 
$$
[\spa (\Fp)/\underline{E}^\times ] \iso \Bun^d.
$$
Let us note $$\mathscr{T}_d$$ for the corresponding universal $\underline{\Q}_p^\times$-torsor over $\Bun^d$. This is the torsor of isomorphisms between $\O(d)$ and the universal line bundle over $\Bun^d$.
Let $\mu (z)=z^k$ and note 
$$
\Hecke_\mu^d= \hd^{-1} (\Bun^d\times \spa (E)^\diamond).
$$
The Hecke correspondence restricts to 
$$
\xymatrix{
 & Hecke_\mu^d \ar[ld]_-{\hg} \ar[rd]^-{\hd} \\
 \Bun^{d+k} && \Bun^d\times \spa (E)^\diamond.
}
$$
Fix a Lubin-Tate group associated to $\pi$ over $\O_E$ and note $E(1)$ its Tate-module and $E(k)=E(1)^{\otimes k}$. This defines a rank $1$ pro-étale local system on $\spa (E)^\diamond$ and thus a $\underline{E}^\times$-torsor $$\LT_k$$
over $\spa (E)^\diamond$. We use the notation $\wedge$ for $\underset{\underline{E}^\times}{\times}$ applied to two $\underline{E}^\times$-torsors. For example one has
$$
\LT_k\wedge \LT_{k'} =\LT_{k+k'}.
$$

\begin{prop}\label{prop:clef cas GL1}
\begin{enumerate}
\item The morphism $\hd$ is an isomorphism.
\item There is a natural isomorphism of $\underline{E}^\times$-torsors
$$
\hg^* \mathscr{T}_{d+k} \simeq \hd^* \mathscr{T}_d \wedge \LT_{-k}
$$
where we still note $\LT_{-k}$ for its pullback from $\spa  (E)^\diamond$ to $\Bun\times \spa (E)^\diamond$.
\end{enumerate}
\end{prop}
\begin{proof}
Point (1) is evident. Point (2) is a consequence of the {\it fundamental exact sequence of $p$-adic Hodge theory}. 
By twisting by powers of  $\O(1)$ one can suppose $d\geq 0$ (twisting by $\O(1)$ induces isomorphisms $\Bun^d\xrig{\sim}\Bun^{d+1}$). Up to taking the dual modification one can suppose $k\geq 0$ too.
There is an exact sequence of sheaves over $\spa (E)^\diamond$
$$
0\ldrt \mathbb{B}^{\ph=\pi^d}\otimes_E E(k)^\diamond \ldrt \mathbb{B}^{\ph=\pi^{k+d}} \ldrt \mathbb{B}_{dR}^+/\Fil^k \mathbb{B}_{dR}\ldrt 0
$$
where we recall that $\mathbb{B}(S)= \GG (S,\O_{Y_S})$.
Now, if $S$ is a perfectoid space and $f:(X_S)_{pro-\et}\drt S_{pro-\et}$, for $\E$ a degree $i$ line bundle on $X_S$, if $\mathscr{T}$ is the corresponding $\underline{E}^\times$-torsor of isomorphisms between $\O(i)$ and $\mathscr{T}$ one has
$$
Rf_* \E = \mathscr{T}\underset{\underline{E}^\times}{\times}  \mathbb{B}^{\ph=\pi^i}.
$$
Since one has $\mathbb{B}^{\ph=\pi^d}\otimes_E E(k)^{\diamond} = \LT_k\underset{\underline{E}^\times}{\times} \mathbb{B}^{\ph=\pi^k}$ one easily deduces that 
$$
\hd^*\mathscr{T}_d = \hg^*\mathscr{T}_{d+k} \wedge \LT_k
$$
whence the result.
\end{proof}

Let now $\ph:W_E\drt \Qlb^\times$ be a continuous character. Define 
$$
\chi:E^\times \underset{\sim}{\xrig{ \ Art \ }} W_E^{ab}\xrig{\ \ph\ }\Qlb^\times
$$
where "$Art$" stands for Artin reciprocity map normalized so that the Frobenius corresponds to a uniformizer. Define now $\mathscr{G}$ to be the rank $1$ $\Qlb$-local system on $\Bun$ such that 
$$
\mathscr{G}_{|\Bun^d} = \mathscr{T}_d\underset{E^\times,\chi}{\times} \Qlb.
$$
{\it This is not still the local system we are looking for.} In fact, according to the character sheaf property, {\it the Weil sheaf we are looking for on $\Bun_{\Fqb}$ does not always descend to $\Bun$} (in our case in this section because $\chi(\pi)$ may not be an $\ell$-adic unit). Thus, define $\F_\ph$ to be the Weil sheaf on $\Bun_{\Fqb}$ such that without its Weil descent datum 
$$
\F_\ph = \mathscr{G}_{|\Bun_{\Fqb}}
$$
and 
{\it its Frobenius descent datum is given by the canonical one of $\mathscr{G}$ times $\chi (\pi)^{-d}$ on $\Bun^d_{\Fpb}$.}

\begin{prop}
The local system $\F_\ph$ satisfies the hypothesis of the conjecture.
\end{prop}
\begin{proof}
Let us begin with the Hecke property. From proposition \ref{prop:clef cas GL1} one deduces that 
\begin{eqnarray} \label{eq:Hecke GL11}
\hd_! \big ( \hg^* \mathscr{G} \big )_{|\Bun^d\times \spa (E)^\diamond}
&=&  \mathscr{T}_d\underset{E^\times}{\times} \LT_{-k} \underset{E^\times, \chi}{\times} \Qlb \\
\nonumber
&=& \mathscr{G}_{|\Bun^d} \boxtimes (\chi\circ \chi_{\LT})^{-k}
\end{eqnarray}
where $\chi_{\LT}:W_E^{ab}\drt \O_E^\times\subset E^\times$ is the Lubin-Tate character. The realization of local class field theory in the torsion points of the Lubin-Tate group tells us that $Art\circ \chi_{\LT}$ satisfies
\begin{enumerate}
\item $Art\circ \chi_{\LT} (\s)=1$ if $\s=Art (\pi)$
\item $Art\circ \chi_{\LT} ( \s)=\s^{-1}$ if $\s$ is in the image of the inertia subgroup of $W_E$ in $W_E^{ab}$.
\end{enumerate}
From point (2) one deduces that 
$$
\hd_! \big ( \hg^*\F_\ph \big )_{\Bun^d_{\Fpb}\times \spa (\breve{E})^\diamond} = \F_{\ph |\Bun^d_{\Fpb}}\boxtimes r_\mu \circ \ph_{|I_E}
$$
without their Frobenius descent datum. 
From equation (\ref{eq:Hecke GL11}) and point (1) one deduces that this is an isomorphic if we equip $r_\mu\circ\ph_{|I_E}$ with the Weil descent datum given by $r_\mu\circ \ph$.
\\

The character sheaf property is checked in the following way. Let $\delta \in E^\times$ with $v_\pi (\delta)=-d$. Then, on $\Bun^d$, we have two $\underline{E}^\times$-torsors, $\mathscr{T}_\delta$ and $\mathscr{T}_d=\mathscr{T}_{\pi^{-d}}$. Consider the $E^\times$-homogeneous space  
$$
\text{Isom} (( L,\pi^{-d}\s), (L,\delta\s)) = 
\{g\in L^\times\ |\ g\delta\s =\pi^{-d}\s g\}.
$$
For such a $g$, $g^\s= \delta\pi^d g$. We deduce from this that 
$$
\mathscr{T}_\delta =\mathscr{T}_d \wedge T_{\delta\pi^d}
$$
where $T_{\delta\pi^d}$ is the $\underline{E}^\times$-torsor over $\Fq$ whose monodromy is given by $Frob_q\mapsto \delta \pi^d$.  The character sheaf property is easily deduced since we "forced it" for $\delta=\pi^{-d}$.
\end{proof}

\subsection{The case of tori}

Let $T$ be a torus over $E$ and $\Bun$ be the corresponding stack. According to Kottwitz
$$
\kappa : B(T)\iso X_*(T)_\Gamma.
$$
There is a decomposition 
$$
\Bun = \coprod_{[\mu]\in X_*(T)_\Gamma}  \Bun^{[\mu]}
$$
where we note $\Bun^{[\mu]} = \{ \kappa=[-\mu]\}$ (we take the opposite of the sign convention in the conjecture like in the $\GL_1$-case, the degree of a vector bundle being the opposite of $\kappa$, maybe we should do this in the entire conjecture).
For $\mu'\in X_*(T)$ we note $\bar{\mu'}\in X_*(T)/\Gamma$ and $[\mu']\in X_*(T)_{\Gamma}$ its class. Given $\mu$ and $\mu'$ there is a Hecke correspondence 
$$
\xymatrix{ & Hecke^{\bar{\mu'}} \ar[ld]_-{\hg} \ar[rd]^-{\hd} \\ 
\Bun^{[\mu+\mu']} &  & \Bun^{[\mu]}\times \spa (E)^\diamond.
}
$$
One checks easily the following proposition.

\begin{prop}
If $E_{\mu'}|E$ is such that $\text{Stab}_{\Gamma} (\mu')=\Gamma_{E_{\mu'}}$ there is an identification $Hecke^{\bar{\mu'}} = \Bun \times  \spa (E_{\mu'})^\diamond$ where $\hd$ is given by $Id\times can_{E_{\mu'}|E}$.
\end{prop}

Let $\ph:W_E\ldrt \,^L T$ be an $L$-parameter. According to local class field theory
$H^1 (W_E,\widehat{T})=\Hom ( T(E),\Qlb)$. We thus have a character
$$
\chi:T(E)\ldrt \Qlb^\times.
$$
For $b\in T(L)$ there is an identification 
$$
x_b :[\spa (\Fqb)/ \underline{T(E)} ]\iso \Bun^{-\kappa (b)}_{\Fqb}.
$$
This defines a $T(E)$-torsor $\mathscr{T}_b$ over $\Bun^{-\kappa (b)}_{\Fqb}$. 
One has $\text{Frob}^* x_b=x_{b^\s}$ and $\text{Frob}^*\mathscr{T}_b=\mathscr{T}_{b^\s}$. 
The $\s$-conjugation $b^\s= b^{-1}.b.b^\s$ defines a natural transformation between $x_b$ and $x_{b^\s}$. It thus defines a Weil descent datum 
\begin{eqnarray}
\label{eq:descent tore}
Weil_b:\text{Frob}^*\mathscr{T}_b \iso \mathscr{T}_b.
\end{eqnarray}
\begin{exem}
If $b\in T(E_n)$ then $\text{Frob}^{n*}x_b=x_b$. This defines an identification between $\text{Frob}^{n*}\mathscr{T}_b$ and $\mathscr{T}_b$ that is to say $\mathscr{T}_b$ is defined over $\Bun_{\mathbb{F}_{q^n}}^{-\kappa ([b])}$. The $n$-th iterate of the descente datum (\ref{eq:descent tore}) is this identification times $ N_{E_n/E} ( b)$. For example, if $b\in T(E)$ this Weil descent datum on $\mathscr{T}_b$ is never effective unless $b$ is a root of unity. 
\end{exem}

If $b'$ is $\s$-conjugate to $b$ let 
$$
T_{b,b'} = \{ g\in T(L)\ | \ gbg^{-\s} =b'\}.
$$
This is an homogenous space under $T(E)$ and we see it as a $T(E)$-torsor on $\spa (\Fqb)$. 
 There is an identification 
$$
\mathscr{T}_b\wedge T_{b,b'} = \mathscr{T}_{b'}
$$
of torsors over $\Bun_{\Fqb}^{-\kappa ([b])}$.
Fix a trivialization of $T_{b,b,'}$ that is to say an element $g\in T(L)$ such that $gbg^{-\s}=b'$. This defines an isomorphism of Weil sheaves 
$$
g:\mathscr{T}_b \iso \mathscr{T}_{b'}
$$
which means $g\circ Weil_b = Weil_{b'}\circ Frob^* g$ where $Frob^* g=g^\s$.
\\

Let $$\ph:W_E\ldrt \,^L T$$ be a Langlands parameter and 
$$
\chi:T(E)\ldrt \Qlb^\times
$$
be the corresponding character given by class field theory, $H^1(W_E,\widehat{T})=\Hom ( T(E),\Qlb^\times)$. 
We define the Weil rank one local system  $\F$ on $\Bun_{\Fqb}$ by the formula 
$$
\F_{|\Bun^{[\mu]}_{\Fqb}} = \mathscr{T}_b \underset{\underline{T(E)},\chi}{\times} \Qlb
$$
for a choice of some $b\in T(L)$ such that $\kappa ([b])= -[\mu]$. Another choice of a $b$ does not change the isomorphism class of our Weil local system (and all the asked properties in the conjecture depend only on the isomorphism class of our Weil perverse sheaf, not on the Weil perverse sheaf itself). 
\\

We now want to prove that $\F$ satisfies the hypothesis of the conjecture. The character sheaf property is immediate. The only thing that we have to prove is the Hecke property.
Consider $\mu,\mu'\in X_*(T)$ and let $E_{\mu'}$ be the field of definition of $\mu'$ as before. Let 
$$
\delta= N_{E_{\mu'}/E_n} ( \mu' (\pi_{\mu'}))\in T(E_n)
$$
where $E_n|E$ is the maximal unramified extension in $E_{\mu'}$. One has $\kappa (\delta ) = [\mu']$ and the pair $(\delta,\mu')$ is admissible. This defines a cristalline representation 
$$
\rho_{\delta,\mu'}:\GG_{E_{\mu'}}^{ab}\ldrt T(E).
$$
Note 
$$
N_{\mu'}: E_{\mu'}^\times \xrig{\ \mu'\ } T(E_{\mu'})\xrig{\ N_{E_{\mu'}/E}\ } T(E).
$$
Composed with Artin reciprocity map $Art:E_{\mu'}^\times \drt \GG_{E_{\mu'}}^{ab}$ this satisfies 
\begin{itemize}
\item $\rho_{\delta,\mu'}\circ Art (x)=N_{\mu'} (x)^{-1}$ for $x\in \O_{E_{\mu'}}^\times$
\item $\rho_{\delta,\mu'}\circ Art  (\pi_{\mu'})=1$.
\end{itemize}

We now note
$$
T_{\delta,\mu'}^{cris}
$$
for the corresponding $\underline{T(E)}$-torsor over $\spa (E_{\mu'})^\diamond$. 

\begin{prop}
For $b\in T(L)$ here is an isomorphism of Weil $\underline{T(E)}$-torsors on $Hecke^{\bar{\mu'}}_{\Fqb}$  
$$
\hg^* \mathscr{T}_{b\delta} \simeq \hd^* \mathscr{T}_b \wedge T_{\delta,\mu'}^{cris}.
$$
\end{prop}
\begin{proof}
Take $(V,\rho) \in \Rep_E (T)$ a linear representation of $T$. Note $\Cp=\widehat{\overline{E}}$. The cristalline representation $\rho\circ \rho_{\delta,\mu'}$ is associated to the filtered isocrystal corresponding to $(V_{E_n},\rho (\delta)\s,\rho\circ \mu')$. This means there is a $\GG_{E_{\mu'}}$-equivariant modification at $\infty\in |X_{\Cp^\flat}|$ of vector bundles (see \cite{Courbe} chap.10)
$$
\big ( V\otimes_E \O_{X_{\Cp^\flat}} \big )_{|X_{\Cp^\flat} \setminus \infty} \iso \E ( V_{E_n},\rho (\delta)\s )_{|X_{\Cp^\flat}\setminus \infty}
$$ 
where the action of $\GG_{E_{\mu'}}$ on the left term is given by $\rho\circ \rho_{\delta,\mu'}$. The type of this modification is given by 
$\rho\circ\mu'$. Twisting by $\rho_*\E_b= \E ( V_L,\rho(b)\s)$ one obtains a $\Gamma$-equivariant 
modification 
$$
\rho_*\big ( \E_b\wedge T_{\delta,\mu'}^{cris}\big )_{|X_{\Cp^\flat}\setminus \infty} \iso \rho_* ( \E_{b\delta})_{|X_{\Cp^\flat\setminus \infty}}. 
$$
Projecting via $(X_{\Cp^{\flat}})_{pro-\et}\drt \spa (\Cp^{\flat})_{pro-\et}$, taking into account the $\Gamma$-action and making $\rho$ vary one obtains the result as in \ref{prop:clef cas GL1}.
\end{proof}

The Hecke property is then proved in the following way. The chocaracter $\mu'$ defines 
$$
\hat{\mu'}\in X^* (\widehat{T})^{\Gamma_{E_{\mu'}}}
$$
which defines a character 
$$
\,^L \mu': \,^L T_{E_{\mu'}}\ldrt \Qlb^\times
$$
where $\,^L T_{E_{\mu'}}= \widehat{T}\rtimes \Gamma_{E_{\mu'}}$ and $\,^L \mu'$ is trivial on the $\Gamma_{E_{\mu'}}$-factor. One then has
$$
r_{\mu'} = \text{Ind}_{E_{\mu'}|E} \,^L \mu'.
$$
The local class field theory isomorphism
$$
H^1 (W_E, \widehat{T}) = \Hom (T(E),\Qlb^\times )
$$ 
is compatible with base change. This means that $\ph_{|W_{E_{\mu'}}}$ corresponds to 
$$
\chi\circ N_{E_{\mu'}/E}:T(E_{\mu'})\ldrt \Qlb^\times.
$$
From this one deduces that 
$$
\chi\circ N_{E_{\mu'}/E}: 
E_{\mu'}^\times \underset{\sim}{\xrig{\ Art\ }} W_{E_{\mu'}}^{ab} \xrig{ \ \,^L \hat{\mu'}
\circ \ph_{W_{E_{\mu'}}}\ } \Qlb^\times .
$$
The Hecke property is deduced from this and the preceding formulas for $\rho_{\delta,\mu'}\circ Art$.

\bibliographystyle{plain}
\bibliography{biblio}
\end{document}